\documentclass[11pt, leqno,twoside]{article}
\usepackage{amssymb}
\usepackage{amsmath}
\usepackage{amsthm}
\usepackage{amsfonts}
\usepackage{color}
 \usepackage{graphicx}

\usepackage[active]{srcltx}

\allowdisplaybreaks

\def\rr{{\mathbb R}}
\def\rn{{{\rr}^n}}

\def\zz{{\mathbb Z}}
\def\nn{{\mathbb N}}

\def\ch{{\mathcal H}}

\def\cx{{\mathcal X}}

\def\bd{{ \mathbb D }}
\def\cd{{ \mathcal D }}

\def\cm{{\mathcal M}}

\def\cb{{\mathcal B}}

\def\ccc{{\mathcal C}}

\def\fz{\infty}
\def\az{\alpha}

\def\lz{\lambda}
\def\dz{\delta}

\def\ez{\epsilon}
\def\kz{\kappa}
\def\bz{\beta}
\def\gz{{\gamma}}

\def\vz{\varphi}

\def\sz{\sigma}

\def\wz{\widetilde}
\def\ls{\lesssim}
\def\gs{\gtrsim}
\def\r{\right}
\def\lf{\left}

\def\diam{{\mathop\mathrm{\,diam\,}}}

\def\dist{{\mathop\mathrm{\,dist\,}}}
\def\loc{{\mathop\mathrm{\,loc\,}}}

\def\lip{{\mathop\mathrm{\,Lip}}}

\def\bint{{\ifinner\rlap{\bf\kern.25em--}
\int\else\rlap{\bf\kern.45em--}\int\fi}\ignorespaces}
\def\dbint{\displaystyle\bint}
\def\bbint{{\ifinner\rlap{\bf\kern.25em--}
\hspace{0.078cm}\int\else\rlap{\bf\kern.45em--}\int\fi}\ignorespaces}

\def\dsup{\displaystyle\sup}

\newtheorem{thm}{Theorem}[section]
\newtheorem{lem}{Lemma}[section]

\newtheorem{rem}{Remark}[section]
\newtheorem{cor}{Corollary}[section]
\newtheorem{defn}{Definition}[section]

\numberwithin{equation}{section}

\begin{document}

\arraycolsep=1pt

\title{\Large\bf  Characterizations of Besov and Triebel-Lizorkin Spaces
on Metric Measure Spaces
\footnotetext{\hspace{-0.35cm}
\noindent{2000 {\it Mathematics Subject Classification:}} 42B35 
\endgraf  {\it Key words and phases:}
Besov space, Triebel-Lizorkin space, Haj\l asz-Besov space, Haj\l asz-Triebel-Lizorkin space,
metric measure space, sharp maximal function
\endgraf Pekka Koskela and
Yuan Zhou were supported by the Academy of Finland grants 120972, 131477.
\endgraf Amiran Gogatishvili was partially supported by
the grant no. 201/08/0383 of the Grant Agency of the Czech Republic, by the   Institutional Research Plan no. AV0Z10190503 of AS CR.
\endgraf $^\ast$ Corresponding author. 
}}
\author{Amiran Gogatishvili,
Pekka Koskela and Yuan Zhou$^\ast$}
\date{ }
\maketitle

\begin{center}
\begin{minipage}{13.5cm}\small
{\noindent{\bf Abstract}\quad On a metric measure space
satisfying the doubling property, we establish several optimal characterizations of
 Besov and Triebel-Lizorkin spaces, including a pointwise characterization.
Moreover, we discuss their (non)triviality  under
a Poincar\'e inequality.
 }
\end{minipage}
\end{center}

\medskip

\section{Introduction\label{s1}}

\hskip\parindent

Let $(\cx,\,d)$ be a metric space and $\mu$ be
a regular Borel measure on $\cx$ such that
all balls defined by $d$
have finite and positive measures, and
assume that $\mu$ satisfies a {\it doubling property}:
there exist constants $C_1>1$ and $n>0$ such that for all $x\in\cx$,
$r\in(0,\,\fz)$ and $\lz\in(1,\,\fz)$,
\begin{equation*}
\mu(B(x,\,\lz r))\le C_1\lz^n\mu(B(x,\,r)).
\end{equation*}

Recall the following definition of Besov spaces from \cite{gks09}.

\begin{defn}\label{d1.1}\rm
Let $s\in(0,\fz)$ and $p,\,q\in(0,\fz]$.
The {\it homogeneous Besov space} $\dot B^s_{p,\,q}(\cx)$ is defined to be  the {\it collection of
all} $u\in L^p_\loc(\cx)$ such that
\begin{equation*}
\|u\|_{\dot B^s_{p,\,q}(\cx)}\equiv\lf(\int_0^\fz\lf(\int_\cx\bint_{B(x,\,t)}
|u(x)-u(y)|^p\,d\mu(y)d\mu(x)\r)^{q/p}
\frac{dt}{t^{1+sq}}\r)^{1/q}<\fz
\end{equation*}
with the usual modification made when $p=\fz$ or $q=\fz$.
\end{defn}

Above, $u\in L^p_\loc(\cx)$ requires that $u\in L^p(B)$ for each ball $B$.

Observe that functions in $\dot B^s_{p,\,q}(\cx)$ have the smoothness of order $s$
as measured by $$t^{-s}\lf(\bint_{B(x,\,t)}|u(x)-u(y)|^p\,d\mu(y)\r)^{1/p}.$$
Recall that, in the literature, there are several ways to measure the smoothness of functions.
For example, letting $s\in[0,\,\fz)$, $\ez\in[0,\,s]$ and $\sz\in(0,\,\fz)$,
 for all measurable functions $u$,  set
\begin{eqnarray*}
 &C^{s,\,\sz}_t(u)(x)\equiv t^{-s}\lf(\dbint_{B(x,\,t)}|u(x)-u(y)|^\sz\,d\mu(y)\r)^{1/\sz},&\\
 &A^{s,\,\sz}_t(u)(x)\equiv t^{-s}\lf(\dbint_{B(x,\,t)}|u(y)-u_{B(x,\,t)}|^\sz\,d\mu(y)\r)^{1/\sz}, &\\
&I^{s,\,\sz}_t(u)(x)\equiv t^{-s}\lf(\displaystyle\inf_{c\in\rr}\dbint_{B(x,\,t)}|u(y)-c|^\sz\,d\mu(y)\r)^{1/\sz},& \\
 &S^{s,\,\ez,\,\sz}_t(u)(x)\equiv t^{(\ez-s)}\dsup_{r\in(0,\,t]}
r^{-\ez}\lf(\inf_{c\in\rr}\dbint_{B(x,\,r)}|u(y)-c|^\sz\,d\mu(y)\r)^{1/\sz}&
\end{eqnarray*}
for all $x\in\cx$ and $t\in(0,\,\fz)$.

The first purpose of this paper is to show that the smoothness of functions
in Besov spaces   can be measured by
the above quantatives with optimal parameters.
To this end,  we introduce the following spaces of Besov type.
In what follows, for our convenience, we denote by $\vec C^{s,\,\sz}$ the operator that maps each $u\in L^\sz_\loc(\cx)$
into a measurable function  $\vec C^{s,\,\sz}(u)$ on $\cx\times(0,\,\fz)$  defined by
$\vec C^{s,\,\sz}(u)(x,\,t)\equiv C^{s,\,\sz}_t(u)(x)$
for all $x\in\cx$ and $t\in(0,\,\fz)$.
We define $\vec A^{s,\,\sz}$, $\vec I^{s,\,\sz}$ and $\vec S^{s,\,\ez,\,\sz}$ analogously.

\begin{defn}\label{d1.2}\rm
Let $s,\,\sz\in(0,\,\fz)$, $\ez\in[0,\,s]$ and
$p,\,q\in(0,\fz]$.
For $\vec E=\vec C^{s,\,\sz},
\vec A^{s,\,\sz}$, $\vec I^{s,\,\sz}$ or $\vec S^{s,\,\ez,\,\sz}$,
  the {\it homogeneous space $\vec E\dot B_{p,\,q}(\cx)$ of Besov type} is defined to be the
{\it collection of all} $u\in L^\sz_\loc(\cx)$ such that
\begin{equation*}
\|u\|_{\vec E\dot B_{p,\,q}(\cx)}\equiv\lf(\int_0^\fz  \|\vec E(u)(\cdot,\,t)\|_{L^p(\cx)}^{q}
\frac{dt}{t}\r)^{1/q}<\fz
\end{equation*}
with the usual modification made when $p=\fz$ or $q=\fz$.
\end{defn}

The main results of this paper read as follows.
For our convenience, for $s\in(0,\,\fz)$ and $p\in(0,\,\fz]$,
we always set    \begin{equation}\label{e1.1}
 p_\ast(s)\equiv \lf\{\begin{array}{ll}np/(n-ps),\quad&  if\ \  p<n/s;\\
 \fz,&  if\ \ p\ge n/s.     \end{array}\r.
\end{equation}

\begin{thm}\label{t1.1} Let $s\in(0,\fz)$ and $p,\,q\in(0,\fz]$.

(i) If $\sz\in(0,\,p]$, then $\dot B^s_{p,\,q}(\cx)=\vec C^{s,\,\sz}\dot B_{p,\,q}(\cx)$.

(ii) If $ \sz \in(0,\,p_\ast(s))$, then $\dot B^s_{p,\,q}(\cx)
 =\vec I^{s,\,\sz}\dot B_{p,\,q}(\cx)$.

(iii) If $\ez\in[0,s)$ and $ \sz \in(0,\,p_\ast(s))  $,
then $\dot B^s_{p,\,q}(\cx)=\vec S^{s,\,\ez,\,\sz}\dot B_{p,\,q}(\cx)$.

(iv) If $p\in(n/(n+s),\,\fz]$ and $ \sz \in(0,\,p_\ast(s))$, then
$\dot B^s_{p,\,q}(\cx)=\vec A^{s,\,\sz}\dot B_{p,\,q}(\cx).$

\noindent Moreover,  the ranges of  $\ez$ and $\sz$  above are optimal in the following sense.

(v) Let $s\in(0,\,1)$, $p\in(0,\,n/s)$ and  $\sz>p_\ast(s)$.
Then there exists a function  $u$ such that for all $q\in(0,\,\fz]$,
$u\in\dot B^s_{p,\,q}(\rn)$  but
$u\notin L^\sz_\loc(\rn)$, and hence, for $\vec E= \vec A^{s,\,\sz}$, $\vec I^{s,\,\sz}$ or $\vec S^{s,\,\ez,\,\sz}$,
$u\notin \vec E\dot B_{p,\,q}(\rn)$.

(vi) Let $p\in(0,\,\,\fz)$, $\sz\in(p,\,\fz)$ and $s\in(0,\,n/p-n/\sz)\cap(0,\,1)$.
Then there exists a function  $u$ such that for all
 $q\in(0,\,\fz]$,
$u\in\dot B^s_{p,\,q}(\rn)$
but $u\notin \vec C^{s,\,\sz}\dot B_{p,\,q}(\rn)$.

(vii) Let $s\in(0,\,1)$ and $p\in(0,\,\fz)$.
Then there exists a function $u\in \dot B^s_{p,\,p}(\rn)$ with $u\notin \vec S^{s,\,s,\,p}\dot B_{p,\,p}(\rn)$.
\end{thm}

We point out that it is natural and necessary to consider the full range of $s$  due to
the nontrivial example of nontrivial Besov spaces $\dot B^s_{n/s,\,n/s}(\cx)$ for all
$s\in(0,\,\fz)$ given by Theorem \ref{t4.3}.

Recently, a fractional pointwise gradient was introduced in
\cite{kyz2} to measure the smoothness of functions.

\begin{defn}\label{d1.3}\rm
Let $s\in(0,\,\fz)$ and let $u$ be a measurable function on $\cx$.
A sequence of nonnegative measurable functions,
${\vec g}\equiv\{g_k\}_{k\in\zz}$, is called
a {\it fractional $s$-Haj\l asz gradient } of $u$ if there exists
 $E\subset\cx$ with $\mu(E)=0$ such that for all $k\in\zz$ and
$x,\, y\in\cx\setminus E$
satisfying $2^{-k-1}\le d(x,\,y)< 2^{-k}$,
 \begin{equation*}
|u(x)-u(y)|\le [d(x,\,y)]^s[g_k(x)+g_k(y)].
\end{equation*}
Denote by $\bd^s(u)$ the {\it collection of all fractional $s$-Haj\l asz
gradients of $u$}.
\end{defn}

In fact,
${\vec g}\equiv\{g_k\}_{k\in\zz}$
above is not really a gradient. One should view it, in the Euclidean
setting (at least when $g_k=g_j$ for all $k,j$), as a maximal function
of the usual gradient.

Our second result  characterizes the Besov spaces in Definition \ref{d1.1}
via the fractional Haj\l asz gradient.
In what follows, for $p,\,q\in(0,\,\fz]$ and a sequence
$\vec g=\{g_k\}_{k\in\zz}$ of nonnegative functions, we always write
$\|\{g_j\}_{j\in\zz}\|_{ \ell^q}\equiv\{\sum_{j\in\zz}|g_j|^q\}^{1/q}$ when $q<\fz$
and $\|\{g_j\}_{j\in\zz}\|_{ \ell^\fz}\equiv \sup_{j\in\zz}|g_j|$,
$\|\{g_j\}_{j\in\zz}\|_{\ell^q(L^p(\cx))}\equiv \|\{\|g_j\|_{L^p(\cx )}\}_{j\in\zz}\|_{ \ell^q}.$

\begin{defn}\rm\label{d1.4}
Let $s\in (0,\fz)$  and $p,\,q\in(0,\,\fz]$.
The {\it homogeneous Haj\l asz-Besov space} $\dot N^s_{p,\,q}(\cx)$ is the
{\it space of all measurable functions} $u$ such that
$$\|u\|_{\dot N^s_{p,\,q}(\cx)}\equiv\inf_{\vec g\in \bd^s (u)}\lf\| \vec g
\r\|_{\ell^q(L^p(\cx))}<\fz. $$
\end{defn}

\begin{thm}\label{t1.2}
Let $s\in(0,\fz)$ and  $p,\,q\in(0,\fz]$.     Then $\dot N^s_{p,\,q}(\cx)=\dot B^s_{p,\,q}(\cx).$
\end{thm}

Under the additional assumptions that $\mu$ also satisfies a reverse doubling condition,
$0<s<1$ and $p>n/(n+1)$, $\dot B^s_{p,\,q}(\cx)$ also allows for a kernel function characterization \cite{my09}. 

Theorem \ref{t1.2} and (i) through (iv) of Theorem \ref{t1.1} follow from Theorem \ref{t2.1} below,
whose proof relies on an inequality of
Poincar\'e type established in Lemma \ref{l2.1}  and a pointwise inequality  given by Lemma \ref{l2.3}.
The proof of (v) through (vii) of Theorem \ref{t1.1} will be given at the end of Section \ref{s2}.

Moreover, in Section 3, we state the corresponding results for Triebel-Lizorkin spaces (see Theorem \ref{t3.1}).
As a special case,
we also establish the equivalence between Haj\l asz-Sobolev spaces
and the Sobolev type spaces of Calder\'on and DeVore-Sharpley (see Corollary \ref{c3.1}).

In Section 4, applying the above characterizations, we prove the triviality of
Besov and Triebel-Lizorkin spaces under a suitable Poincar\'e inequality
(see Theorem \ref{t4.1} and Theorem \ref{t4.2}),
and also give some examples of nontrivial Besov and Triebel-Lizorkin
spaces to show the ``necessity'' of such a Poincar\'e inequality (see Theorem \ref{t4.3}).

Finally, we make some conventions. Throughout the paper,
we denote by $C$ a {\it positive
constant} which is independent
of the main parameters, but which may vary from line to line.
Constants with subscripts, such as $C_0$, do not change
in different occurrences. The {\it notation} $A\ls B$ or $B\gs A$
means that $A\le CB$. If $A\ls B$ and $B\ls A$, we then
write $A\sim B$.
For two spaces $X$ and $Y$ endowed with (semi-)norms, the notation $X\subset Y$ means that
$u\in X$ implies that $u\in Y$ and $\|u\|_Y\ls\|u\|_X$,
and the notation $X= Y$ means that $X\subset Y$ and $Y\subset X$.
Denote by $\zz$ the {\it set of integers} and $\nn$ the {\it set of positive integers}.
For any locally integrable function $f$,
we denote by $\bbint_E f\,d\mu$ the {\it average
of $f$ on $E$}, namely, $\bbint_E f\,d\mu\equiv\frac 1{\mu(E)}\int_E f\,d\mu$.

\section{Proofs of Theorem \ref{t1.1} and Theorem  \ref{t1.2} \label{s2}}

We begin with a Poincar\'e type inequality.

\begin{lem}\label{l2.1}
Let $s\in(0,\,\fz)$ and $p\in(0,\,n/s)$.
Then for every pair of $\ez,\ez'\in(0,\,s)$ with $\ez<\ez'$,
there exists a positive constant
$C$ such that for all  $x\in\rn$, $k\in\zz$,
measurable functions $u$ and $\vec g\in \bd^s(u)$,
\begin{eqnarray*}
&&\inf_{c\in\rr}\lf(\dbint_{B(x,\,2^{-k})}
    \lf|u(y)-c\r|^{p_\ast(\ez)}\,d\mu(y)\r)^{1/p_\ast(\ez)}\\
 &&\quad\le C  2^{-k\ez'}\sum_{j\ge k-2}2^{-j(s-\ez')}\lf\{\dbint_{B(x,\,2^{-k+1})}
    [g_j(y)]^p\,d\mu(y)\r\}^{1/p},\nonumber
\end{eqnarray*}
where $p_\ast(\ez)$ is as in \eqref{e1.1}.
\end{lem}

Recall that when $s\in(0,\,1]$ and $\cx=\rn$, Lemma \ref{l2.1} was established in \cite[Lemma 2.3]{kyz2}.
Generally, Lemma \ref{l2.1} can be proved by an argument
similar to that of \cite[Lemma 2.3]{kyz2} with the aid of
the following variant of \cite[Theorem 8.7]{h03}.
In what follows, for every $s\in (0,\fz)$ and measurable function $u$ on $\cx$, a non-negative
function $g$ is called an {\it $s$-gradient} of $u$
if there exists a set $E\subset \cx$ with $\mu(E)=0$ such that for all
$x,\ y\in\cx\setminus E$,
\begin{equation} \label{e2.1}
|u(x)-u(y)|\le [d(x,y)]^s[g(x)+g(y)].
\end{equation}
Denote by $\cd^s(u)$ the {\it collection of all $ s$-gradients of $u$}.

\begin{lem}\label{l2.2}
Let $s\in(0,\,\fz)$, $p\in(0,\,n/s)$
and let $p_\ast(s)$ be as in \eqref{e1.1}. Then
there exists a positive constant
$C$ such that for all $x\in\cx$, $r\in(0,\,\fz)$,
and all measurable functions $u$ and $g\in \cd^s(u)$,
\begin{eqnarray*}
   \inf_{c\in\rr}\lf(\dbint_{B(x,\, r)}
    \lf|u(y)-c\r|^{p_\ast(s)}\,d\mu(y)\r)^{1/p_\ast(s)}
 \le C  r^s\lf(\dbint_{B(x,\, 2r)}
    [g(y)]^{p}\,d\mu(y)\r)^{1/p}.
\end{eqnarray*}
\end{lem}

When $s\in(0,\,1]$, since $d^s$ is also a distance on $\cx$,
 Lemma \ref{l2.2} follows from \cite[Theorem 8.7]{h03}.
When $s\in(1,\,\fz)$, with $p<n/s$ in mind, checking the proof of
\cite[Theorem 8.7]{h03} line by line, we still have
 Lemma \ref{l2.2}. We omit the details.

We still need the   following pointwise inequality,
which is a variant of the pointwise inequality  established in \cite[(5.7)]{kyz2}.

\begin{lem}\label{l2.3}
Let $\sz\in(0,\,\fz)$. Then there exists a positive constant $C$
such that, for each function $u\in L^\sz_\loc(\cx)$,
one can find a set $E$ with $\mu(E)=0$ so that
 for each pair of points
$x,\,y\in \cx\setminus E$ with $d(x,\,y)\in[2^{-k-1}, 2^{-k})$,
\begin{eqnarray} \label{e2.2}
|u(x)-u(y)| &&\le C \sum_{j\ge k-2} \lf\{\inf_{c\in\rr}\lf[ \bint_{B(x,\,2^{-j})}|u(w)-c|^\sz \,dw\r]^{1/\sz }\r.\\
&&\quad\quad\quad\quad\quad\quad\lf.
+\inf_{c\in\rr}\lf[ \bint_{B(y,\,2^{-j})}|u(w)-c|^\sz \,dw\r]^{1/\sz }\r\}.\nonumber
\end{eqnarray}
\end{lem}

To prove Lemma \ref{l2.3}, we need  Lemma \ref{l2.4} below.
In what follows, for a real-valued measurable function $u$ and
a ball $B$,  define the {\it median value} of $u$ on $B$ by
\begin{equation}\label{e2.3}
 m_u(B)\equiv\max\lf\{a\in\rr,\,\mu(\{x\in B:\ u(x)<a\})\le \frac{\mu(B)}2\r\} .
\end{equation}

\begin{lem}\label{l2.4}
For every real-valued measurable function $u$,
there exists a measurable set $E\subset \cx$ with $\mu(E)=0$ such that for  all $z\in\cx\setminus E$,
$$u(z)=\lim_{\mu(B)\to0,\,B\ni z}m_u(B).$$
\end{lem}

Lemma \ref{l2.4} was  proved in \cite[Lemma 2.2]{f90} for $\cx=\rn$, and
the very same argument  gives Lemma \ref{l2.4}.
We omit the details.

\begin{proof}[Proof of Lemma \ref{l2.3}]
Let $u$ be a real-valued measurable function and $E$ be the set given by Lemma \ref{l2.4}.
Then for all $z\in\cx\setminus E$, by Lemma \ref{l2.4},
$m_u(B(z,\,2^{-j}))\to u(z)$ as $j\to\fz$, and hence
\begin{eqnarray*}
&&\lf|u(z)-m_u(B(z,\,2^{-k}))\r|\\
&&\quad\le \sum_{j\ge k}\lf|m_u(B(z,\,2^{-j}))-m_u(B(z,\,2^{-j-1}))\r|\\
 &&\quad\le \sum_{j\ge k}\lf[\lf|m_u(B(z,\,2^{-j}))-c_{B(z,\,2^{-j})}\r|
+\lf|m_u(B(z,\,2^{-j-1}))-c_{B(z,\,2^{-j} )}\r|\r],
\end{eqnarray*}
where $c_{B(z,\,2^{-j} )}$ is a real number such that
\begin{equation*}
 \bint_{B(z,\,   2^{-j}  )}\lf|u(w)-c_{B(z,\,2^{-j} )}\r|^\sz \,d\mu(w)\le
2\inf _{c\in\rr} \bint_{B(z,\,  2^{-j} )}|u(w)-c|^\sz \,d\mu(w).
\end{equation*}

We claim  that  for every ball $B$ and each $c\in\rr$,
\begin{equation}\label{e2.4}
|m_u(B)-c|\le  \lf\{2\bint_B|u(w)-c|^\sz \,d\mu(w)\r\}^{1/\sz }.
\end{equation}
Assume that this claim holds for a moment.
We have
\begin{eqnarray}\label{e2.5}
\lf |m_u(B(z,\,2^{-j}))-c_{B(z,\,2^{-j})}\r|
&&\le \lf\{ 2\bint_{B(z,\,2^{-j})}|u(w)-c_{B(z,\,2^{-j})}|^\sz \,d\mu(w)\r\}^{1/\sz }\\
&&\ls\inf_{c\in\rr}\lf\{ \bint_{B(z,\,2^{-j})}|u(w)-c|^\sz \,d\mu(w)\r\}^{1/\sz } \nonumber
\end{eqnarray}
and
\begin{eqnarray}\label{e2.6}
\lf |m_u(B(z,\,2^{-j-1}))-c_{B(z,\, 2^{-j})}\r|&&\le
\lf\{2\bint_{B(z,\, 2^{-j-1})}|u(w)-c_{B(z,\,2^{-j})}|^\sz \,dw\r\}^{1/\sz }\\
&&\ls
\lf\{\bint_{B(z,\, 2^{-j})}|u(w)-c_{B(z,\,2^{-j})}|^\sz \,dw\r\}^{1/\sz }\nonumber\\
&&\ls\inf_{c\in\rr}\lf\{ \bint_{B(z,\, 2^{-j})}|u(w)-c|^\sz \,dw\r\}^{1/\sz }.\nonumber
\end{eqnarray}
Therefore,
\begin{equation}\label{e2.7}
 \lf|u(z)-m_u(B(z,\,2^{-k}))\r|
 \ls \sum_{j\ge k}\inf_{c\in\rr}\lf\{ \bint_{B(z,\,r_j)}|u(w)-c|^\sz \,dw\r\}^{1/\sz }.
\end{equation}

For $x,\,y\in\cx\setminus E$ with $2^{-k-1}\le d(x,\,y)< 2^{-k}$, we write

\begin{eqnarray*}
|u(x)-u(y)|&&\le  |u(x)-m_u(B(x,\,2^{-k+1}))|+|m_u(B(x,\,2^{-k+1}))-c_{B(x,\,2^{-k+1})}|
\\
&&\quad +|c_{B(x,\,2^{-k+1})}-m_u(B(y,\,2^{-k}))|
+|u(y)-m_u(B(y,\,2^{-k}))|.
\end{eqnarray*}
By an argument similar to that of \eqref{e2.6},
we have
$$
 |c_{B(x,\,2^{-k+1})}-m_u(B(y,\,2^{-k}))|
\ls\inf_{c\in\rr}\lf\{ \bint_{B(x,\, 2^{-k+1})}|u(w)-c|^\sz \,dw\r\}^{1/\sz },$$
which together with \eqref{e2.7} and \eqref{e2.5} gives \eqref{e2.2}.

Now we prove the claim \eqref{e2.4}.
For every ball $B$ and each $c\in\rr$,  observing that
$m_{u-c}(B)=m_u(B)-c$ and  recalling that
$|m_{u}(B)|\le m_{|u|}(B)$ as proved in
\cite[Lemma 2.1]{f90}, we have $  |m_u(B)-c|\le m_{|u-c|}(B)$.
By this, \eqref{e2.4} is reduced to
\begin{equation}\label{e2.9}
m_{|u-c|}(B)\le \lf\{2\bint_B|u(w)-c|^\sz \,d\mu(w)\r\}^{1/\sz }.
\end{equation}
To see this, letting $\dz\equiv \bint_B|u(w)-c|^\sz\,dw$,
by Chebyshev's inequality, for every $a>2$, we have
\begin{eqnarray*}
\mu\lf(\lf\{w\in B:\ |u(w)-c|\ge(a\dz)^{1/\sz }\r\}\r)
&&=\mu\lf(\lf\{w\in B:\ |u(w)-c|^\sz\ge a\dz \r\}\r)\\
&&\le ( a\dz)^{-1 }\int_B|u(w)-c|^\sz\,dw<\frac{\mu(B)}2,
\end{eqnarray*}
 which yields that 
$$\mu\lf(\lf\{w\in B:\ |u(w)-c|<(a\dz)^{1/\sz}\r\}\r)>\frac{\mu(B)}2$$
and hence by \eqref{e2.3}, 
$m_{|u-c|}(B)\le (a\dz)^{1/\sz}.$
Then letting $ a \to 2$, we obtain \eqref{e2.9} and hence prove the claim \eqref{e2.4}.
This finishes the proof of Lemma \ref{l2.3}.
\end{proof}

We also use the following lemma.

\begin{lem}\label{l2.5}
Let $s,\,\sz\in(0,\,\fz)$, $\ez\in[0,\,s]$ and
$p,\,q\in(0,\fz]$.
Let $\vec E=\vec C^{s,\,\sz},
\vec A^{s,\,\sz}$, $\vec I^{s,\,\sz}$ or $\vec S^{s,\,\ez,\,\sz}$.
Then for each measurable function $u$,
\begin{equation}\label{e2.10}
\|u\|_{\vec E\dot B_{p,\,q}(\cx)}\sim \lf\|\lf\{\vec E(u)(x,\,2^{-k})\r\}_{k\in\zz}\r\|_{\ell^q(L^p(\cx))}.
\end{equation}
\end{lem}

\begin{proof}
Observe that $\vec E(u)(x,\,t)\ls\vec E(u)(x,\,2^{-k+1}) $
for all $t\in (2^{-k},\,2^{-k+1}]$ and $x\in\cx$,
from which   \eqref{e2.10} follows by a simple computation.
This finishes the proof of Lemma \ref{l2.5}.
\end{proof}

With the aid of Lemma \ref{l2.1}, Lemma \ref{l2.3} and Lemma \ref{l2.5},
we obtain the following result,  which, together with the fact
$\vec C^{s,\,p}\dot B_{p,\,q}(\cx)=\dot B^s_{p,\,q}(\cx)$,
implies Theorem \ref{t1.2} and (i) through (iv) of Theorem \ref{t1.1}.

\begin{thm}\label{t2.1}
Let $s\in(0,\fz)$ and  $p,\,q\in(0,\fz]$.

(i) If $\sz\in(0,\,p]$, then  $\dot N^s_{p,\,q}(\cx)=\vec C^{s,\,\sz}\dot B_{p,\,q}(\cx)$.

(ii) If $\sz\in(0,\,p_\ast(s))$, then $\dot N^s_{p,\,q}(\cx)
 =\vec I^{s,\,\sz}\dot B_{p,\,q}(\cx)$.

(iii) If $\ez\in[0,\,s)$ and $\sz\in(0,\,p_\ast(s))$,
then $\dot N^s_{p,\,q}(\cx) = \vec S^{s,\,\ez,\,\sz}\dot B_{p,\,q}(\cx).$

(iv) If $p\in(n/(n+s),\,\fz]$ and $\sz\in(0,\,p_\ast(s))$,
 then
$\dot N^s_{p,\,q}(\cx)=\vec A^{s,\,\sz}\dot B_{p,\,q}(\cx).$
\end{thm}

\begin{proof}
First, notice that if $\mu(\cx)<\fz$, then
\begin{equation}\label{e2.11}
 \diam\cx\equiv\sup_{x,\,y\in\cx}{d(x,\,y)}<\fz.
\end{equation}
Indeed, suppose that $\diam\cx=\fz$. Fix a ball $B(x_0,\,r_0)\subset \cx$. By our assumptions on $\mu$, we have
$\mu(B(x_0,\,r_0))>0$. Notice that for any $x_1\in\cx$ with  $d(x_1,\,x_0)\ge 2r_0$, by the doubling property
and $B(x_0,\,r_0)\subset B(x_1,\,2d(x_1,\,x_0))$,
we have
$$\mu(B(x_1,\,\frac12d(x_1,\,x_0)))\ge (C_1)^{-1}4^{-n}\mu(B(x_1,\,2d(x_1,\,x_0)))\ge (C_1)^{-1}4^{-n}\mu(B(x_0,\,r_0)).$$
Let $r_1=2d(x_1,\,x_0)$. Since $B(x_1,\,\frac12d(x_1,\,x_0))\cap B(x_0,\,r_0)=\emptyset$,
we have $$\mu(\cx)> \mu(B(x_1,\,r_1))\ge [1+(C_1)^{-1}4^{-n}]\mu(B(x_0,\,r_0)).$$
Repeating this procedure for $N$ times, we can find
 $x_N\in\cx$  and $r_N>0$ such that
\begin{eqnarray*}
 \mu(\cx)&&>\mu(B(x_N,\,r_N))\ge [1+(C_1)^{-1}4^{-n}]\mu(B(x_{N-1},\,r_{N-1}))\\
&&\ge\cdots \ge [1+(C_1)^{-1}4^{-n}]^N\mu(B(x_0,\,r_0)),
\end{eqnarray*}
which tends to infinity as $N\to\fz$. This is a contradiction.
Thus $\diam\cx<\fz$.

Assume that  $2^{-k_0-1}\le\diam\cx< 2^{-k_0}$ for some $k_0\in\zz$.
Observe that
$$\|u\|_{\vec E\dot B_{p,\,q}(\cx)}
\sim \lf(\sum_{k\ge k_0-2} \lf\|\vec E(u)(\cdot,\,2^{-k})\r\|^q_{L^p(\cx)}\r)^{1/q}$$
and that for any $\vec g\in\bd^s(u)$, we can always take $g_k\equiv0$ for $k< k_0-2$.
Because of this, the proof of Theorem \ref{t1.2} for the case $\mu(\cx)<\fz$ is a slight modification of that
for the case $\mu(\cx)=\fz$ below.  In what follows, we only consider the case $\mu(\cx)=\fz$.

We first prove (ii) and (iii). Observing that
\begin{equation}\label{e2.12}
I^{s,\,\sz}_t(u)(x)\le S^{s,\,\ez,\,\sz}_t(u)(x)
\end{equation}
 for all $t\in(0,\,\fz)$ and $x\in\cx$,
we have $\vec S^{s,\,\ez,\,\sz}\dot B_{p,\,q}(\cx)\subset \vec I^{s,\,\sz}\dot B_{p,\,q}(\cx)$.
So it suffices to prove that $\vec I^{s,\,\sz}\dot B_{p,\,q}(\cx)\subset \dot N^s_{p,\,q}(\cx)\subset \vec S^{s,\,\ez,\,\sz}\dot B_{p,\,q}(\cx)$.

To prove $\vec I^{s,\,\sz}\dot B_{p,\,q}(\cx)\subset \dot N^s_{p,\,q}(\cx)$,
let $u\in \vec I^{s,\,\sz}\dot B_{p,\,q}(\cx)$
and $E$ with $\mu(E)=0$ be as in Lemma \ref{l2.3}. By Lemma \ref{l2.3}, it is easy to see  that
for $x,\,y\in \rn\setminus E$ and $d(x,\,y)\in[2^{-k-1}, 2^{-k})$,
\begin{equation} \label{e2.13}
|u(x)-u(y)| \le C[d(x,\,y)]^s \sum_{j\ge k-2}2^{(k-j)s}[I^{s,\,\sz}_{2^{-j}}(u)(x)+
I^{s,\,\sz}_{2^{-j}}(u)(y)].
\end{equation}
For $k\in\zz$, set
\begin{equation}\label{e2.14}
g_k\equiv \sum_{j\ge k-2}2^{(k-j)s}I^{s,\,\sz}_{2^{-j}}(u).
\end{equation}
Then $\vec g\equiv\{g_k\}_{k\in\zz}\in\bd^s(u)$ modulo a fixed constant and it is easy to check that
$$\|\vec g\|_{\ell^q(L^p(\cx))}
\ls\|\{ I^{s,\,\sz}_{2^{-k}}(u)\}_{k\in\zz}\|_{\ell^q(L^p(\cx))};$$
see the proof of \cite[Theorem 2.1]{kyz2} for details.
So, by Lemma \ref{l2.5}, $u\in \dot N^s_{p,\,q}(\cx)$ and
\begin{equation}\label{e2.15}
 \|u\|_{\dot N^s_{p,\,q}(\cx)}\le \|\vec g\|_{\ell^q(L^p(\cx))}
\ls\|\{ I^{s,\,\sz}_{2^{-k}}(u)\}_{k\in\zz}\|_{\ell^q(L^p(\cx))}\sim \|u\|_{\vec I^{s,\,\sz}\dot B_{p,\,q}(\cx)}.
\end{equation}
This leads to $\vec I^{s,\,\sz}\dot B_{p,\,q}(\cx)\subset \dot N^s_{p,\,q}(\cx)$.

To prove that
$\dot N^s_{p,\,q}(\cx)\subset \vec S^{s,\,\ez,\,\sz}\dot B_{p,\,q}(\cx)$,
since $\sz<p_\ast(s)$, we can choose $\ez'\in(0,\,s)$ and
$\dz\in(0,\,p)$ such that
$\sz\le \dz_\ast(\ez')=n\dz/(n-\ez'\dz)$.
We also let
$\ez''\in(\ez',\,s)$ and $\ez'''\in(0,\, \min\{s-\ez'',\,s-\ez\})$.
For given $u\in \dot N^s_{p,\,q}(\cx)$, take $\vec g\in\bd^s(u)$ with
$\|\vec g\|_{\ell^q(L^p(\cx))}\le 2\|u\|_{\dot N^s_{p,\,q}(\cx)}$.
Set $$h_k\equiv 2^{k\ez'''}\sum_{i\ge k}2^{-i\ez'''}g_i$$ for $k\in\zz$.
Then $\vec h\equiv\{h_k\}_{k\in\zz}\in\bd^s(u)$,
 $h_i\le 2^{(i-k)\ez'''}h_k$
for any $i\ge k$, and moreover, it is easy to check
\begin{equation}\label{e2.16}
\|\vec h\|_{\ell^q(L^p(\cx))}\ls \|\vec g\|_{\ell^q(L^p(\cx))}\ls \|u\|_{\dot N^s_{p,\,q}(\cx)};
\end{equation}
see the proof of \cite[Theorem 2.1]{kyz2} for details.
Then by Lemma \ref{l2.1}, for all $j\in\zz$ and $j\ge k$,
\begin{eqnarray}\label{e2.17}
I_{2^{-j}}^{\ez,\,\sz}(u)(x)&&= 2^{j\ez}\lf(\inf_{c\in\rr}\bint_{B(x,\,2^{-j})}|u(z)-c|^\sz\,d\mu(z) \r)^{1/\sz} \\
&&\le 2^{j\ez}\lf(\inf_{c\in\rr}\bint_{B(x,\,2^{-j})}|u(z)-c|^{n\dz/(n-\ez'\dz)}\,d\mu(z) \r)^{(n-\ez'\dz)/n\dz}\nonumber\\
&&\ls 2^{j\ez} 2^{-j\ez''}\sum_{i\ge j-2}2^{-i(s-\ez'')}\lf(\bint_{B(x,\,2^{-j+1})}[h_i(z)]^{\dz}\,d\mu(z)\r)^{1/\dz}\nonumber\\
&&\ls 2^{j(\ez-s)}\sum_{i\ge j-2}2^{(j-i)(s-\ez'')}\lf(\bint_{B(x,\,2^{-j+1})}[h_i(z)]^{\dz}\,d\mu(z)\r)^{1/\dz}\nonumber\\
&&\ls 2^{j(\ez-s)}\sum_{i\ge j-2}2^{(j-i)(s-\ez'')}\cm_{\dz}(h_i)(x)\nonumber\\
&&\ls 2^{j(\ez-s)}\sum_{i\ge j-2}2^{(j-i)(s-\ez'')}2^{(i-k)\ez'''}\cm_{\dz}(h_k)(x)\nonumber\\
&&\ls 2^{j(\ez-s)}2^{(j-k)\ez'''}\cm_{\dz }(h_k)(x).\nonumber
\end{eqnarray}
Here and in what follows $\cm$ denotes the {\it Hardy-Littlewood maximal operator} and
$\cm_\dz(u)\equiv[\cm(|u|^\dz)]^{1/\dz}$ for all $u\in L^\dz_\loc(\cx)$ and $\dz\in(0,\,\fz)$.
Thus for all $k\in\zz$,
\begin{eqnarray}\label{e2.18}
S_{2^{-k}}^{s,\,\ez,\,\sz}(u)(x)&&\ls2^{k(s-\ez)}\sup_{j\ge k}2^{j(\ez-s)} I^{\ez,\,\sz}_{2^{-j}}(u)(x)\\
 &&\ls\sup_{j\ge k}2^{-(j-k)(s-\ez)}2^{(j-k)\ez'''}\cm_{\dz}(h_k)(x)\ls \cm_{\dz}(h_k)(x).\nonumber
\end{eqnarray}
So, by the $L^{p/\dz}(\cx)$-boundedness of $\cm$, Lemma \ref{l2.5} and \eqref{e2.16},
we have $u\in \vec S ^{s,\,\ez,\,\sz}\dot B_{p,\,q}(\cx)$
and
 \begin{eqnarray*}
\|u\|_{\vec S^{s,\,\ez,\sz}\dot B_{p,\,q}(\cx)}&&\ls\| \{S_{2^{-k}}^{s,\,\ez,\,\sz}(u)\}_{k\in\zz}\|_{\ell^q(L^p(\cx))}\\
&&\ls
\|\{\cm_{\dz}(h_k)\}_{k\in\zz}\|_{\ell^q(L^p(\cx))}\ls \|\vec h\|_{\ell^q(L^p(\cx))} \ls\|u\|_{\dot N^s_{p,\,q}(\cx)}.
 \end{eqnarray*}
This yields $\dot N^s_{p,\,q}(\cx)\subset \vec S^{s,\,\ez,\,\sz}\dot B_{p,\,q}(\cx)$ and thus finishes the proofs of
(ii) and (iii).

Now we prove (i).  Since
\begin{equation}\label{e2.19}
I^{s,\,\sz}_t(u)(x)\le C^{s,\,\sz}_t(u)(x)
\end{equation}
for all $t\in(0,\,\fz)$ and $x\in\cx$,
 we have
$\vec C^{s,\,\sz}\dot B_{p,\,q}(\cx)\subset \vec I^{s,\,\sz}\dot B_{p,\,q}(\cx)$,
and hence by (ii),
$\vec C^{s,\,\sz}\dot B_{p,\,q}(\cx)\subset \dot N^s_{p,\,q}(\cx)$.
So we only need to show that $\dot N^s_{p,\,q}(\cx)\subset \vec C^{s,\,\sz}\dot B_{p,\,q}(\cx)$.
For given $u\in \dot N^s_{p,\,q}(\cx)$, take $\vec g\in\bd^s(u)$ with
$\|\vec g\|_{\ell^q(L^p(\cx))}\le 2\|u\|_{\dot N^s_{p,\,q}(\cx)}$.
Then by Lemma \ref{l2.1}, for all $k\in\zz$,
\begin{eqnarray*}
C_{2^{-k}}^{s,\,\sz}(u)(x)&&= 2^{ks}\lf(\sum_{j\ge k} \frac1{\mu(B(x,\,2^{-k}))}
\int_{B(x,\,2^{-j})\setminus B(x,\,2^{-j-1})}|u(z)- u(x)|^\sz\,d\mu(z) \r)^{1/\sz} \\
&&\le 2^{ks}\lf(\sum_{j\ge k-2}  2^{-js\sz}
\bint_{B(x,\,2^{-j})}\lf([g_j(z)]^\sz+[g_j(x)]^\sz\r)\,d\mu(z) \r)^{1/\sz}\\
&&=  \lf(\sum_{j\ge k} 2^{-(j-k)s\dz}[g_j(x)]^\sz+ \sum_{j\ge k} 2^{-(j-k)s\sz}
\bint_{B(x,\,2^{-j})}[g_j(z)]^\sz\,d\mu(z) \r)^{1/\sz}.
\end{eqnarray*}
If $p>\sz$, then  when $\sz\in(0,\,1)$, applying the
H\"older inequality, we have
\begin{eqnarray}\label{e2.20}
 C_{2^{-k}}^{s,\,\sz}(u)(x)&&\ls \lf(\sum_{j\ge k} 2^{-(j-k)s\sz}[\cm_\sz(g_j)(x)]^\sz\r)^{1/\sz}\\
&&\ls  \sum_{j\ge k} 2^{-(j-k)s/2}\cm_\sz(g_j)(x)\lf(\sum_{j\ge k} 2^{-(j-k)s\sz/2(1-\sz)}\r)^{(1-\sz)/\sz}\nonumber\\
&&\ls
\sum_{j\ge k} 2^{-(j-k)s/2}
\cm_\sz(g_j)(x),\nonumber
\end{eqnarray}
and  when $\sz\in[1,\,p)$, by $1/\sz\le1$,
 $$C_{2^{-k}}^{s,\,\sz}(u)(x)\ls\sum_{j\ge k} 2^{-(j-k)s}
\cm_\sz(g_j)(x).$$

From this, it is easy to deduce that
$$\|\{ C_{2^{-k}}^{s,\,\sz}(u)\}_{k\in\zz}\|_{\ell^q(L^p(\cx))}\ls
\|\{\cm_\sz(g_k)\}_{k\in\zz}\|_{\ell^q(L^p(\cx))}.$$
By this, the $L^{p/\sz}(\cx)$-boundedness of $\cm$ and Lemma \ref{l2.5},
we have $u\in \vec C^{s,\,\sz}\dot B_{p,\,q}(\cx)$ and
$$\|u\|_{\vec C^{s,\,\sz}\dot B_{p,\,q}(\cx)}\ls\|\{ C_{2^{-k}}^{s,\,\sz}(u)\}_{k\in\zz}\|_{\ell^q(L^p(\cx))}\ls
\|\vec g\|_{\ell^q(L^p(\cx))}
\ls\|u\|_{\dot N^s_{p,\,q}(\cx)}.$$
If $\sz=p$, then
\begin{eqnarray*}
 \|u\|_{\vec C^{s,\,p}\dot B_{p,\,q}(\cx)}&&\ls
 \lf\{\sum_{k\in\zz}\lf(\int_\cx \sum_{j\ge k} 2^{-(j-k)sp}[g_j(x)]^p\,d\mu(x)\r)^{q/p}\r\}^{1/q}\\
&&\quad\quad+ \lf\{\sum_{k\in\zz}\lf(\int_\cx\sum_{j\ge k} 2^{-(j-k)sp}
\bint_{B(x,\,2^{-j})}[g_j(z)]^p\,d\mu(z) \,d\mu(x)\r)^{q/p}\r\}^{1/q}\\
&&\ls
 \lf\{\sum_{k\in\zz}\lf(\int_\cx \sum_{j\ge k} 2^{-(j-k)sp}[g_j(x)]^p\,d\mu(x)\r)^{q/p}\r\}^{1/q}\\
&&\ls \|\vec g\|_{\ell^q(L^p(\cx))}\ls\|u\|_{\dot N^s_{p,\,q}(\cx)}.
\end{eqnarray*}
This   gives  $\dot N^s_{p,\,q}(\cx)\subset \vec C^{s,\,\sz}\dot B_{p,\,q}(\cx)$ and thus finishes the proof of  (i).

Finally, we prove (iv). Trivially,
\begin{equation}\label{e2.21}
 I^{s,\,\sz}_t(u)(x)\le   A^{s,\,\sz}_t(u)(x)
\end{equation}
for all $x\in\cx$ and $t\in(0,\,\fz)$, which implies that
$\vec A^{s,\,\sz}\dot B_{p,\,q}(\cx)\subset \vec I^{s,\,\sz}\dot B_{p,\,q}(\cx)$.
On the other hand, since
  $p> n/(n+s)$ and $\sz<p_\ast(s)$, we can find $\sz'\in(\max\{\sz,\,1\},\, p_\ast(s))$.
Notice that, for any $c\in\rr$, by the Minkowski inequality and the H\"older inequality,
\begin{eqnarray*}
 \lf(\bint_{B(x,\,t)}|u-u_{B(x,\,t)}|^{\sz'}\,d\mu(z)\r)^{1/\sz'}
&&\le \lf(\bint_{B(x,\,t)}|u-c|^{\sz'}\,d\mu(z) \r)^{1/\sz'}
+|c-u_{B(x,\,t)}|\\
&&\le2\lf(\bint_{B(x,\,t)}|u-c|^{\sz'}\,d\mu(z) \r)^{1/\sz'},
\end{eqnarray*}  which together with the H\"older inequality again implies that
\begin{equation}\label{e2.22}
 A^{s,\,\sz}_t(u)(x) \le  A^{s,\,\sz'}_t(u)(x)\le 2  I^{s,\,\sz'}_t(u)(x)
\end{equation}
for all $u\in L^\sz_\loc(\cx)$, $x\in\cx$ and $t\in(0,\,\fz)$.
Then $\vec I^{s,\,\sz'}\dot B_{p,\,q}(\cx)\subset\vec A^{s,\,\sz}\dot B_{p,\,q}(\cx)$.
Recall that
we have proved that
$\vec I^{s,\,\sz'}\dot B_{p,\,q}(\cx)=\dot N^s_{p,\,q}(\cx)=\dot B^s_{p,\,q}(\cx)=\vec I^{s,\,\sz}\dot B_{p,\,q}(\cx)$.
So we obtain (iv).
The proof of Theorem \ref{t2.1} is finished.
\end{proof}

\begin{rem}\label{r2.3}\rm
For the theory of Besov spaces on $\rn$ see, for example, \cite{t83,p76,t92},
and, on metric measure spaces, see \cite{hmy08,my09,gks09}.
When $s\in(0,\,1)$, $p\in(n/(n+s),\,\fz]$
and $q\in(0,\,\fz]$,
the equivalence $\dot B^s_{p,\,q}(\cx)=\vec C^{s,\,1}\dot B_{p,\,q}(\cx)$
 on a metric measure space satisfying both a
doubling and a reverse doubling property was established in \cite{my09} by M\"uller and Yang
and the pointwise characterization of
$\dot B^s_{p,\,q}(\cx)$  in \cite{kyz2}.
But Theorem \ref{t2.1} (hence Theorem \ref{t1.1} and Theorem \ref{t1.2})
does not require a reverse doubling property and  also
works for the whole index range.
\end{rem}

One can derive the following inequality from the proof of \eqref{e2.17}.

\begin{cor}\label{c2.1}
For $s\in(0,\,\fz)$, $\dz\in(0,\,\fz)$, $\sz\in(0,\,\dz_\ast(s))$  and  $\sz'\in(0,\,\fz)$,
there exist $\ez>0$ satisfying $\sz<\dz_\ast(s-\ez)$, and constant $C$ such that for all $u\in L^\sz_\loc(u)$,
$k\in\zz$  and $x\in\cx$,
\begin{equation} \label{e2.23}
 I_{2^{-k}}^{s,\,\sz}(u)(x)\le C 2^{-k\ez}\cm_{\dz}\lf(\sum_{j\ge k-2}I^{s-\ez,\,\sz'}_{2^{-j}}(u)\r)(x).
\end{equation}
\end{cor}

\begin{proof}
If $\sz'\ge\sz$, then \eqref{e2.23} is trivial or follows from the H\"older inequality.
If $\sz'<\sz$, then we  employ the argument for \eqref{e2.17} with the
special choice $g_j\equiv I_{2^{-j}}^{s,\,\sz'}(u)$ for all $j\ge k-2$.
We leave the details to the reader.
\end{proof}

We  close Section \ref{s2} by proving the optimality of
 the ranges of  $\ez$ and $\sz$ in Theorem \ref{t1.1}.

\begin{proof}[Proofs of (v) though (vii) of Theorem \ref{t1.1}.]
(v) For $\az\in(0,\,\fz)$, define
$$
u_\az(x)\equiv|x|^{-\az}\chi_{B(0,\,1)}(x)+ \chi_{\rn\setminus B(0,\,1)}(x).
$$
We first claim that
 $u_\az\in\dot B^s_{p,\,q}(\rn)$ when  $\az\in(0,\,n/p-s)$.
To see this,
for $j\le 0$, we set
$$g_j(x)\equiv 2^{js}|x|^{-\az}\chi_{B(0,\,1)}(x),$$
and, for $j\ge 1$, we set
$$g_j(x)\equiv 2^{js}|x|^{-\az}\chi_{B(0,\,2^{-j-3})}(x)+
2^{-j(1-s)}|x|^{-\az-1}\chi_{B(0,\,1)\setminus  {B(0,\,2^{-j-3})}}(x).$$
Then it is easy to check that $\vec g\equiv\{g_j\}_{j\in\zz}\in\bd^s(u_\az)$ modulo a fixed constant.
Moreover, since $p\az <n$, for $j\le0$, we have that
\begin{eqnarray*}
\|g_j\|^p_{L^p(\rn)}\le \int_{B(0,\,1)} 2^{jps}|x|^{-p\az}\,dx\ls2^{jps},
\end{eqnarray*}
and
for $j\ge1$,
\begin{eqnarray*}
\|g_j\|^p_{L^p(\rn)}&&\le\int_{B(0,\,2^{-j-3})} 2^{jps}|x|^{-p\az}\,dx\\
&&\quad+
\int_{B(0,\,1)\setminus  {B(0,\,2^{-j-3})}}2^{-jp(1-s)}|x|^{-p(\az+1) }\,dx\\
&&\ls 2^{j[p(s+\az)-n]}+2^{-jp(1-s)}.
\end{eqnarray*}
Observing that
 $s+\az<n/p$ and recalling that $s<1$,
we have $\|\vec g\|_{\ell^q(L^p(\rn))}<\fz$, which is as desired.
Now, taking $\az= n/\sz$ and noticing $s+\az<s+n/p_\ast(s)=n/p$,
we have $u_\az\in\dot N^s_{p,\,q}(\rn)$ and hence by Theorem \ref{t1.2}, 
$u_\az\in\dot B^s_{p,\,q}(\rn)$.
This yields (v) since $u_\az\notin L^\sz_\loc(\rn)$.

(vi)  Let $\az= n/\sz$. Since $\az+s<n/p$, as shown in (v),  $u_\az\in\dot N^s_{p,\,q}(\rn)$. 
Let us check that $\|u_\az\|_{\vec C^{s,\,\sz}\dot B_{p,\,q}(\rn)}=\fz$.
Indeed, if $1/2\le|x|\le 3/4$, then
$B(0,\,1/4)\subset B(x,\,1)$ and for all $z\in B(0,\,1/4)$,
$$|x|^{-\az}\le (1/2)^{-\az}= 2^{-\az} (1/4)^{-\az}\le 2^{-\az} |z|^{-\az},$$
which implies that
$|u(x)-u(z)|\ge(1-2^{-\az})|z|^{-\az}$, and hence
by $n=\az\sz$,
$$\bint_{B(x,\,1)}|u(x)-u(z)|^\sz\,dz\gs\bint_{B(0,\,1/4)}|z|^{-n}\,dz=\fz.$$
Therefore
$$\|u\|_{\vec C^{s,\,\sz}\dot B_{p,\,q}(\rn)}\gs  \lf\{\int_{B(0,\,3/4)\setminus B(0,\,1/2)}
\lf(\bint_{B(x,\,1)}|u(x)-u(z)|^\sz\,dz\r)^{p/\sz}\,dx \r\}^{1/p}=\fz$$
as desired. This gives (vi).

(vii) Let $\az= n/p-s$ and $\bz\in(-2/p,\,-1/p)$, and define
$$u(x)\equiv |x|^{-\az}\lf(\log\frac{2}{|x|}\r)^\bz \chi_{B(0,\,1)}(x)+ \chi_{\rn\setminus B(0,\,1)}(x).$$
We  claim that $u\in \dot N^s_{p,\,p}(\rn)$.
To see this, similarly to (v),
for $j\le 0$, we set
$$g_j(x)\equiv 2^{js}|x|^{-\az}\lf(\log\frac{2}{|x|}\r)^\bz\chi_{B(0,\,1)}(x),$$
and for $j\ge 1$, we set
\begin{eqnarray*}
 g_j(x)\equiv&&2^{js}|x|^{-\az}\lf(\log\frac{2}{|x|}\r)^\bz\chi_{B(0,\,2^{-j-3})}(x)\\
&&+
2^{-j(1-s)}|x|^{-\az-1}\lf(\log\frac{2}{|x|}\r)^\bz\chi_{B(0,\,1)\setminus  {B(0,\,2^{-j-3})}}(x).
\end{eqnarray*}
Then  $\vec g\equiv\{g_j\}_{j\in\zz}\in\bd^s(u)$ modulo a fixed constant.
Since $\az+s=n/p$ and $p\bz<-1$, we have that
\begin{eqnarray*}
\sum_{j\ge1}\|g_j\|^p_{L^p(\rn)}&&\ls\sum_{j\ge1}\int_{B(0,\,2^{-j-3})} 2^{jps}|x|^{-p\az}\lf(\log\frac{2}{|x|}\r)^{p\bz}\,dx\\
&&\quad\quad+\sum_{j\ge1}\int_{B(0,\,1)\setminus  {B(0,\,2^{-j-3})}}2^{-jp(1-s)}|x|^{-p(\az+1) }\lf(\log\frac{2}{|x|}\r)^{p\bz}\,dx\\
%&&=\int_{B(0,\,1)} \sum_{2^{-j-3}\ge|x|}2^{jps}|x|^{-p\az}\lf(\log\frac{2}{|x|}\r)^{p\bz}\,dx\\
%&&\quad\quad+\int_{B(0,\,1)}\sum_{2^{-j-3}\le|x|}2^{-jp(1-s)}|x|^{-p(\az+1) }\lf(\log\frac{2}{|x|}\r)^{p\bz}\,dx\\
&&\ls\int_{B(0,\,1)} |x|^{-p(\az+s)}\lf(\log\frac{2}{|x|}\r)^{p\bz}\,dx\\
&&\ls \int_0^1\lf(\log\frac{2}t\r)^{p\bz}\frac{dt}t<\fz
\end{eqnarray*}
and that
$$
\sum_{j\le0}\|g_j\|^p_{L^p(\rn)}\ls\int_{B(0,\,1)} |x|^{-p(\az+s)}\lf(\log\frac{2}{|x|}\r)^{p\bz}\,dx<\fz. $$
Thus
$u\in\dot N^s_{p,\,p}(\rn)$. On the other hand, for any $x\in B(0,\,1/2)$ and all $t>|x|$,
\begin{eqnarray*}
S_t^{s,\,s,\,p}(u)(x)&&\ge |x|^{-s}\lf(\inf_{c\in\rr}\bint_{B(x,\,|x|)}|u(z)-c|^p\,dx\r)^{1/p}\\
&&\ge (2|x|)^{-s}\lf(\bint_{B(x,\,|x|)}|u(z)-u(x_0)|^p\,dx\r)^{1/p},
\end{eqnarray*}
where may choose $x_0\in B(x,\,|x|)$.
Moreover, up to a rotation, we can assume that $x_0=x|x_0|/|x|$. Observe that if $|x_0|\ge |x|$, then
for $z\in B(x/2,\,|x|/4)\subset B(0,\,3|x|/4)$,
$$|u(z)-u(x_0)|\ge  |u(3x/4)-u(x)|\gs u(x).$$
Moreover, if $|x_0|\le |x|$, then for $z\in B(3x/2,\,|x|/4)\subset B(0,\,1)\setminus B(0,\,5|x|/4)$,
$$|u(z)-u(x_0)|\ge  |u(5x/4)-u(x)|\gs u(x).$$
Hence by $B(x/2,\,|x|/4)\cup B(3x/2,\,|x|/4)\subset B(x,\,|x|)$ and $\az+s=n/p$, we have
 \begin{eqnarray*}
S_t^{s,\,s,\,p}(u)(x)&&\gs |x|^{-\az-s}\lf(\log\frac{2}{|x|}\r)^\bz\sim|x|^{-n/p}\lf(\log\frac{2}{|x|}\r)^\bz,
\end{eqnarray*}
from which together with $p\bz+1>-1$, it follows that
 \begin{eqnarray*}
\|u\|^p_{\vec S^{s,\,s,\,p}\dot B_{p,\,p}(\rn)}
&&\gs\sum_{j\ge0}\|S_{2^{-j}}^{s,\,s,\,p}(u)\|^p_{L^p({B(0,\,2^{-j-1})})}\\
&&\gs \sum_{j\ge0}\int_{B(0,\,2^{-j-1})}|x|^{-n}\lf(\log\frac{2}{|x|}\r)^{p\bz}\,dx\\
&&\sim  \int_{B(0,\,1)} |x|^{-n}\lf(\log\frac{2}{|x|}\r)^{p\bz+1}\,dx\\
&&\sim\int_0^{1/2}\lf(\log\frac{2}{t}\r)^{p\bz+1}\frac{dt}t =\fz.
\end{eqnarray*}
The proof of Theorem \ref{t2.1} is finished.
\end{proof}

\section{Triebel-Lizorkin spaces\label{s3}}

In what follows, for $p,\,q\in(0,\,\fz]$ and a sequence $\vec g$ of measurable functions,
we set  $\lf\| \vec g
\r\|_{L^p(\cx,\,\ell^q)}\equiv\lf\|  \| \vec g\|_{\ell^q}
\r\|_{L^p(\cx )}$.

\begin{defn}\rm\label{d3.1}
Let
$s\in (0,\fz)$  and $p,\,q\in(0,\,\fz]$.
The {\it homogeneous Haj\l asz-Triebel-Lizorkin space} $\dot M^s_{p,\,q}(\cx)$ is the {\it
space of all measurable functions} $u$ such that $\|u\|_{\dot M^s_{p,\,q}(\cx)} <\fz,$
where when $p\in(0,\,\fz)$ or $p,\,q=\fz$,
$$\|u\|_{\dot M^s_{p,\,q}(\cx)}\equiv\inf_{\vec g\in \bd^s (u)}\lf\| \vec g
\r\|_{L^p(\cx,\,\ell^q)},$$
and when $p=\fz$ and $q\in(0,\,\fz)$,
$$\|u\|_{\dot M^s_{\fz,\,q} (\cx)}\equiv
\inf_{\vec g\in \bd^s (u)}\sup_{k\in\zz}\sup_{x\in\cx}
\lf\{\sum_{j\ge k}\bint_{B(x,\,2^{-k})} [g_j(y)]^q\,d\mu(y)\r\}^{1/q}.
 $$
\end{defn}

\begin{defn}\label{d3.2}\rm
Let $s,\,\sz\in(0,\,\fz)$, $\ez\in[0,\,s]$ and
$p,\,q\in(0,\fz]$.
For $\vec E=\vec C^{s,\,\sz},
\vec A^{s,\,\sz}$, $\vec I^{s,\,\sz}$ or $\vec S^{s,\,\ez,\,\sz}$,
  the {\it homogeneous space} $\vec E\dot F_{p,\,q}(\cx)$ of Triebel-Lizorkin type is defined to be the
{\it collection of all } $u\in L^\sz_\loc(\cx)$ such that $
\|u\|_{\vec E\dot F_{p,\,q}(\cx)} <\fz
$, where when $p\in(0,\,\fz)$,
\begin{equation*}
\|u\|_{\vec E\dot F_{p,\,q}(\cx)}\equiv\lf\| \lf(\int_0^\fz  [\vec E(u)(\cdot,\,t)]^{q}\frac{dt}{t}
\r)^{1/q}\r\|_{L^p(\cx)}
\end{equation*}
with the usual modification made when  $q=\fz$,
 and when $p=\fz$ and $q\in(0,\,\fz)$,
\begin{equation*}
\|u\|_{\vec E\dot F_{p,\,q}(\cx)}\equiv \sup_{x\in\cx}\sup_{r>0}  \lf(\int_0^r \bint_{B(x,\,r)}  [\vec E(u)(y,\,t)]^{q}\,d\mu(y)\frac{dt}{t}
\r)^{1/q}
\end{equation*}
and when $p,\,q=\fz$, $\|u\|_{\vec E\dot F_{\fz,\,\fz}(\cx)}\equiv\|u\|_{\vec E\dot B_{\fz,\,\fz}(\cx)}$.
\end{defn}

\begin{thm}\label{t3.1}
Let $s\in(0,\fz)$ and  $p,\,q\in(0,\fz]$.  Let $r\equiv\min\{p,\,q\}$.

(i) If $\sz\in(0,\,r)$, then  $\dot M^s_{p,\,q}(\cx)=\vec C^{s,\,\sz}\dot F_{p,\,q}(\cx)$.

(ii) If $\sz\in(0,\,r_\ast(s))$, then $\dot M^s_{p,\,q}(\cx)
 =\vec I^{s,\,\sz}\dot F_{p,\,q}(\cx)$.

(iii) If $\ez\in[0,\,s)$ and $\sz\in(0,\,r_\ast(s))$,
then $\dot M^s_{p,\,q}(\cx) = \vec S^{s,\,\ez,\,\sz}\dot F_{p,\,q}(\cx).$

(iv) If $r\in(n/(n+s),\,\fz]$ and $\sz\in(0,\,r_\ast(s))$,
 then
$\dot M^s_{p,\,q}(\cx)=\vec A^{s,\,\sz}\dot F_{p,\,q}(\cx).$
\end{thm}

\begin{proof}
The proof of Theorem \ref{t3.1} is similar to
that of Theorem \ref{t2.1}. We only sketch it.

By \eqref{e2.12}, we have
$\vec S^{s,\,\ez,\,\sz}\dot F_{p,\,q}(\cx)\subset \vec I^{s,\,\sz}\dot F_{p,\,q}(\cx)$.

For $u\in \vec I^{s,\,\sz}\dot F_{p,\,q}(\cx)$,
by taking $\vec g\equiv\{g_k\}_{k\in\zz}$ as in \eqref{e2.14}, similarly to   \eqref{e2.15},
we can show that
$\|u\|_{\dot M^s_{p,\,q}(\cx)}\ls \|u\|_{\vec I^{s,\,\sz}\dot F_{p,\,q}(\cx)}$.
Hence, $\vec I^{s,\,\sz}\dot F_{p,\,q}(\cx)\subset \dot M^s_{p,\,q}(\cx)$.

The result $\dot M^s_{p,\,q}(\cx)\subset \vec S^{s,\,\ez,\,\sz}\dot F_{p,\,q}(\cx)$
follows from  an  argument similar to  that for  $\dot N^s_{p,\,q}(\cx)\subset \vec S^{s,\,\ez,\,\sz}\dot B_{p,\,q}(\cx)$,
where the inequality \eqref{e2.18} plays an important role.
The restriction $\sz\in(0,\,r_\ast(s))$ ensures the existence of $\dz\in(0,\,r)$
and $\ez'\in(0,\,\ez)$ such that $\sz\le\dz_\ast(\ez')$.
Moreover, by $\dz\in(0,\,r)$, we can use
 the Fefferman-Stein maximal inequality (see \cite{gly08}) to obtain
$$\|\{\cm_{\dz}(h_k)\}_{k\in\zz}\|_{L^p(\cx,\,\ell^q)}\ls \|\vec h\|_{L^p(\cx,\,\ell^q)}.$$
This gives   (ii) and (iii).

For (i), by \eqref{e2.19}, we  have
$\vec C^{s,\,\sz}\dot F_{p,\,q}(\cx)\subset\vec I^{s,\,\sz}\dot F_{p,\,q}(\cx)\subset \dot M^s_{p,\,q}(\cx)$.
The converse result $\dot M^s_{p,\,q}(\cx)\subset\vec C^{s,\,\sz}\dot F_{p,\,q}(\cx)$
follows from \eqref{e2.20} and an argument similar to the proof of
$\dot N^s_{p,\,q}(\cx)\subset\vec C^{s,\,\sz}\dot B_{p,\,q}(\cx)$ for $\sz\in(0,\,p)$.
Here the restriction $\sz\in(0,\,r)$ comes from
the Fefferman-Stein maximal inequality  used to prove
$$\|\{\cm_{\sz}(g_k)\}_{k\in\zz}\|_{L^p(\cx,\,\ell^q)}\ls \|\vec g\|_{L^p(\cx,\,\ell^q)}.$$
This gives (i).

For (iv), the equivalence $\vec A^{s,\,\sz}\dot F_{p,\,q}(\cx)=\dot M^s_{p,\,q}(\cx)$
follows from
\eqref{e2.21}, \eqref{e2.22} with $\sz'\in (\max\{\sz,\,1\},\,r_\ast(s))$ and (ii).
This gives (iv) and hence finishes the proof of Theorem \ref{t3.1}.
\end{proof}

Notice that, in Theorem \ref{t3.1} (iii), we have the restriction $\ez\in[0,\,s)$.
However, when $\ez=s$ and $q=\fz$, we have the following result.

\begin{thm}\label{t3.2}
Let $s\in(0,\fz)$ and  $p\in(0,\fz]$.
If $\sz\in(0,\,p_\ast(s))$,
then $\dot M^s_{p,\,\fz}(\cx) = \vec S^{s,\,s,\,\sz}\dot F_{p,\,\fz}(\cx)$.
\end{thm}

 \begin{proof}
To see $\vec S^{s,\,s,\,\sz}\dot F_{p,\,\fz}(\cx)\subset \dot M^s_{p,\,\fz}(\cx) $, let $u\in \vec S^{s,\,s,\,\sz}\dot F_{p,\,\fz}(\cx)$.
By \eqref{e2.13} and taking $\vec g\equiv\{g_k\}_{k\in\zz}$ with
$g_k=S^{s,\,s,\,\sz}_{2^{-k+2}}(u)$, we have $\vec g\in\bd^s(u)$
and $$\|\vec g\|_{L^p(\cx,\,\ell^\fz)}\le \|\{S^{s,\,s,\,\sz}_{2^{-k}}(u)\}_{k\in\zz}\|_{L^p(\cx,\,\ell^\fz)}\sim\|u\|_{\vec S^{s,\,s,\,\sz}\dot F_{p,\,\fz}(\cx)},$$
which implies that
$u\in \dot M^{s}_{p,\,\fz}(\cx)$ and
$\|u\|_{\dot M^{s}_{p,\,\fz}(\cx)}\ls \|u\|_{\vec S^{s,\,s,\,\sz}\dot F_{p,\,\fz}(\cx)}$.

Conversely, let $u\in \dot M^{s}_{p,\,\fz}(\cx)$ and $\vec g\in\bd^s(u)$ with
$\|\vec g\|_{L^p(\cx,\,\ell^\fz)}\le 2  \|u\|_{\dot M^{s}_{p,\,\fz}(\cx)}$.
Taking $g\equiv\sup_{k\in\zz} g_k=\|\vec g\|_{ \ell^\fz}$, we have $g\in\cd^s(u)$ and
$\|g\|_{L^p(\cx)}\ls \|u\|_{\dot M^{s}_{p,\,\fz}(\cx)}$.
By $\sz\in(0,\,p_\ast(s))$, let  $\dz\in(0,\,p)$ such that $\sz=\dz_\ast(s)$.
Then by Lemma \ref{l2.2}, for all $x\in\cx$ and $k\in\zz$,
$$S^{s,\,s,\,\sz}_{2^{-k+2}}(u)(x)\le \cm_\dz(g)(x),$$
which together with  the $L^{p/\dz}(\cx)$-boundedness of $\cm$
implies that $u\in \vec S^{s,\,s,\,\sz}\dot F_{p,\,\fz}(\cx)$
and
$\|u\|_{\vec S^{s,\,s,\,\sz}\dot F_{p,\,\fz}(\cx)}\ls\|u\|_{\dot M^{s}_{p,\,\fz}(\cx)}$.
This finishes the proof of Theorem \ref{t3.2}.
\end{proof}

Let $s,\,\sz\in(0,\,\fz)$ and recall  the classical {\it fractional sharp maximal functions}
$$u^{\sharp,\,s}(x)\equiv\sup_{t\in(0,\,\fz)} t^{-s}\bint_{B(x,\,t)}|u(z)-u_{B(x,\,t)}|\,d\mu(z)$$
and $$u^{\sharp,\,s}_\sz(x)\equiv\sup_{t\in(0,\,\fz)} t^{-s}\inf_{c\in\rr}\lf(\bint_{B(x,\,t)}|u(z)-c|^\sz\,d\mu(z)\r)^{1/\sz}.$$
The {\it Sobolev-type space} $\dot\ccc^{s,\,p}(\cx)$ of Calder\'on and DeVore-Sharpley is defined as
the {\it collection of all locally integrable functions $u$}
such that $\|u\|_{\dot\ccc^{s,\,p}(\cx)}\equiv\|u^{\sharp,\,s}\|_{L^p(\cx)}<\fz;$
see \cite{ds84,s07}.
Also observe that
$$u^{\sharp,\,s}_\sz(u)(x)=  \sup_{t\in(0,\,\fz)} S^{s,\,s,\,\sz}_t(u)(x)\sim
\|\vec S^{s,\,s,\sz}(u)(x)\|_{\ell^\fz},$$
and hence $\|u\|_{\vec S^{s,\,s,\,\sz}\dot F_{p,\,\fz}(\cx)}=\|u^{\sharp,\,s}_\sz(u)\|_{L^p(\cx)}.$
On the other hand, recall from \cite{kyz2} that $\dot M^s_{p,\,\fz}(\cx)$ is simply the {\it Haj\l asz-Sobolev space} $\dot M^{s,\,p}(\cx)$. Here $\dot M^{s,\,p}(\cx)$ is the {\it collection of all functions $u$} such that
 $$\|u\|_{\dot M^{s,\,p}(\cx)}\equiv\inf_{g\in\cd^s(u)}\|g\|_{L^p(\cx)}<\fz,$$
where $\cd^s(u)$ is the {\it set of all $s$-gradients of $u$} as in \eqref{e2.1}.
Then, as a consequence of Theorem \ref{t3.1} and Theorem \ref{t3.2}, we have the following corollary.

\begin{cor}\label{c3.1} Let $s\in(0,\,\fz)$ and $p\in(0,\,\fz]$.

(i) If $\sz\in(0,\,p_\ast(s))$, then
$u\in \dot M^{s,\,p}(\cx)$ if and only if $u^{\sharp,\,s}_\sz\in L^p(\cx)$,
and  moreover, for every $u\in \dot M^{s,\,p}(\cx)$,
$\|u\|_{\dot M^{s,\,p}(\cx)}\sim  \|u^{\sharp,\,s}_\sz(u)\|_{L^p(\cx)}$.

(ii) If  $p\in(n/(n+s),\,\fz]$, then $\dot \ccc^{s,\,p}(\cx)=\dot M^{s,\,p}(\cx)$.
\end{cor}

\begin{rem}\label{r3.1}\rm
Notice that $\dot M^s_{p,\,p}(\cx)=\dot B^s_{p,\,p}(\cx)$ for all $s\in(0,\,\fz)$ and $p\in(0,\fz]$.
Then (v) through  (vii) of Theorem \ref{t1.1} also give the optimality of the ranges of $\sz$ and $\ez$
in Theorem \ref{t3.1} at least for the spaces $\dot M^s_{p,\,p}(\rn)$.
\end{rem}

\begin{rem}\label{r3.2}\rm

Recall that
Calder\'on \cite{c72}  characterized Sobolev spaces $W^{1,\,p}(\rn)$ for $p\in(1,\,\fz]$
via the fractional sharp maximal function  $u^{\sharp,\,1}$ and
Miyachi \cite{m90} characterized Hardy-Sobolev spaces $H^{1,\,p}(\rn)$ for $p\in(n/(n+1),\,1]$.
For the theory of Sobolev-type spaces $\dot\ccc^{s,\,p}(\rn)$ and their equivalence with 
the usual Triebel-Lizorkin spaces
$\dot F^s_{p,\,\fz}(\rn)$, see \cite{ds84,c84}.
Shvartsman \cite{s07} introduced the spaces $\dot\ccc^{s,\,p}(\cx)$ on metric measure spaces
for $s\in(0,\,1]$ and $p\in(1,\,\fz]$.
See \cite{t92} and the references therein for
the fractional sharp maximal function characterization of general Triebel-Lizorkin spaces
on $\rn$.

The pointwise Sobolev spaces on metric measure spaces were
introduced and the pointwise characterization of $W^{1,\,p}(\rn)$ for $p\in(1,\,\fz]$
was established by Haj\l asz \cite{h96,h03}.
The pointwise characterization for the Hardy-Sobolev spaces $H^{1,\,p}(\rn)$
for $p\in(n/(n+1),\,1]$ was given by Koskela and Saksman \cite{ks08}.
Yang \cite{y03} established a pointwise characterization for the
Triebel-Lizorkin spaces $\dot F^{s}_{p,\,\fz}(\cx)$ for $s\in(0,\,1)$
and $p\in(1,\,\fz]$ on Ahlfors regular spaces.
For the pointwise characterization   $\dot M^s_{p,\,q}(\cx)=\dot F^{s}_{p,\,q}(\cx)$
on a metric measure space satisfying doubling and reverse doubling properties,
when $s\in(0,\,1)$ and $ p,\,q\in(n/(n+s),\,\fz]$,
see \cite{kyz09,kyz2};  also see \cite{hmy08,g02,my09} for some other characterizations.
Above, the spaces $\dot F^s_{p,\,q}(\cx)$ are defined using kernel functions. 

By Theorem \ref{t3.1}, Theorem \ref{t3.2} and Corollary \ref{c3.1},
  Haj\l asz-Triebel-Lizorkin spaces appear to be a natural substitute of Triebel-Lizorkin spaces
on a metric measure space satisfying the doubling property.
\end{rem}

\section{Triviality and nontriviality\label{s4}}

We say that $\cx$ supports a {\it weak $(1,\,p)$-Poincar\'e inequality}
with $p\in[1,\,\fz)$  if
there exist positive constants $C$ and $\lz>1$ such that for all
functions $u$,  $p$-weak upper gradients $g$ of $u$ and balls $B$ with radius $r>0$,
$$\bint_B|u(x)-u_B|\,d\mu(x)\le C r
\lf\{\bint_{\lz B}[g(x)]^p\,d\mu(x)\r\}^{1/p}.$$
Recall that a nonnegative Borel function $g$ is called a {\it $p$-weak upper gradient
of $u$} if
\begin{equation}\label{e4.1}
|u(x)-u(y)|\le \int_\gz g\,ds
\end{equation}
for all $\gz\in\Gamma_{\rm rect}\setminus \Gamma$,
where $x$ and $y$ are the endpoints of $\gz$,
$\Gamma_{\rm rect}$ denotes the collection of non-constant compact rectifiable  curves
and $\Gamma$ has $p$-modulus zero.
If $\cx$ is complete, the above Poincar\'e inequality holds if and only if
it holds for each Lipschitz function with the pointwise Lipschitz constant
$$\lip (u)(x)=\limsup_{r\to0}\sup_{y\in B(x,\,r)}\frac{|u(x)-u(y)|}r$$
on the right-hand side. See \cite{hk98} for more details.

By triviality of $\dot N^s_{p,\,q}(\cx)$ or $\dot M^s_{p,\,q}(\cx)$
below we mean that they only contain constant functions.
In order to obtain such a conclusion, one needs some connectivity assumption on $\cx$;
simply consider $B(0,\,1)\cup B(x_0,\,1)$ where $x_0\in\rn$ and $|x_0|>3$, 
equipped with the Euclidean distance and Lebesgue measure. Then 
$\chi_{B(0,\,1)}\in \dot M^s_{p,\,q}(\cx)\cap\dot N^s_{p,\,q}(\cx)$ for all $s,\,p,\,q$.
Notice that  $\cx$ does not support any Poincar\'e inequality.

\begin{thm}\label{t4.1}
Suppose that $\cx$  supports a
weak $(1,\,p)$-Poincar\'e inequality with $p\in(1,\,\fz)$.
Then for all $q\in(0,\,\fz)$, $\dot N^1_{p,\,q}(\cx)$ and $\dot M^1_{p,\,q}(\cx)$ are trivial.
\end{thm}

\begin{proof}
Since for $q\in(0,\,p)$, $\dot M^1_{p,\,q}(\cx)\subset \dot M^1_{p,\,p}(\cx)$
and $\dot N^1_{p,\,q}(\cx)\subset \dot N^1_{p,\,p}(\cx)=\dot M^1_{p,\,p}(\cx)$,
we only need to prove that for $q\in[p,\,\fz)$, $\dot M^1_{p,\,q}(\cx)$ and $\dot N^1_{p,\,q}(\cx)$  are trivial.
Assume that $q\in[p,\,\fz)$. Notice that
$\dot M^1_{p,\,q}(\cx)\subset \dot M^1_{p,\,\fz}(\cx)=\dot M^{1,\,p}(\cx)$,
where $\dot M^{1,\,p}(\cx)$ is the Haj\l asz-Sobolev space \cite{h96}.
Moreover, under the weak $(1,\,p)$-Poincar\'e inequality,
it is known that
$\dot M^{1,\,p}(\cx)=\dot N^{1,\,p}(\cx)$
(see \cite[Theorem 4.9]{s00} and \cite{kz08}), where
$\dot N^{1,\,p}(\cx)$ is the Newtonian Sobolev space introduced in \cite{s00}.
So $\dot M^1_{p,\,q}(\cx)\subset  \dot N^{1,\,p}(\cx)$.
Let $u\in \dot M^1_{p,\,q}(\cx)$. Then $u\in \dot N^{1,\,p}(\cx)$.
The proof of the trivialilty of $\dot M^1_{p,\,q}(\cx)$  is reduced to proving $\|u\|_{\dot N^{1,\,p}(\cx)}=0$.
To this end, it suffices to  find a sequence $\{\rho_k\}_{k\in\nn}$ of $p$-weak upper gradients
of $u$ such that $\|\rho_k\|_{L^p(\cx)}\to 0$ as $k\to\fz$.

For $k\in\nn$, set
$$\rho_k(x)\equiv   \sup_{j\ge k} 2^{j}\bint_{B(x,\,2^{-j})}|u(z)-u_{B(x,\,2^{-j})}| \,d\mu(z).
$$
Then $\rho_k$ is nonnegative Borel measurable function
for all $k\in\nn$. 
Moreover, we have that $\lim_{k\to\fz}\rho_k(x)=0$ for almost all $x\in\cx$. 
Indeed, by a discrete variant of Theorem \ref{t3.1} (ii),
$$ \lf\| \lf\{ I^{1,\,1}_{2^{-j}}(u) \r\}_{j\in\zz}\r \|_{L^p(\cx,\,\ell^q)} 
\sim \|u\|_{\dot M^{1}_{p,\,q}(\cx)}<\fz,$$
which implies that 
$ \| \{ I^{1,\,1}_{2^{-j}}(u)(x)  \}_{j\in\zz}\|_{\ell^q}<\fz$
and hence 
 $\rho_k(x)\le  \|\{  I^{1,\,1}_{2^{-j}}(u)(x) \}_{j\ge k}\|_{\ell^q} \to0 
$
as $k\to\fz$
for almost all $x\in\cx$. 
Moreover, applying the Lebesgue dominated convergence theorem, we have 
$\|\rho_k\|_{L^p(\cx)}\to0$ as $k\to\fz$.

Now it suffices to check that $\rho_k$ is a $p$-weak upper gradient
of $u$.
Observe that if $\rho_k(x)<\fz$, then $\lim_{j\to\fz}u_{B(x,\,2^{-j})}$ exists.
In fact,  we have
$$|u_{B(x,\,2^{-j})}-u_{B(x,\,2^{-\ell})}|\ls 2^{-\min\{j,\,\ell\} } \rho_k(x)\to 0$$
as $j,\,\ell \to \fz$.
For such an $x$, we define
$\wz u(x)\equiv\lim_{j\to\fz}u_{B(x,\,2^{-j})}$.
Generally, for $x\in\cx$, if
$\lim_{j\to\fz} u_{B(x,\,2^{-j})}$ exists,  then we define
$\wz u(x)\equiv\lim_{j\to\fz} u_{B(x,\,2^{-j})}$; otherwise,
 $\wz u(x)\equiv0$. Obviously, $u(x)=\wz u(x)$ for almost all $x\in\cx$,
and hence  $u$ and $\wz u$ generate the same element of $\dot N^{1,\,p}(\cx)$.
Therefore we only need to check that $\rho_k$ is a $p$-weak upper gradient
of $\wz u$.
To this end, notice that for all $x,\,y\in\cx$ with $d(x,\,y) \le 2^{-k-2}$,
we have
$$|\wz u(x)-\wz u(y)|\le d(x,\,y)[\rho_k(x)+\rho_k(y)].$$
Moreover, by \cite[Proposition 3.1]{s00}, $u$ is absolutely continuous on $p$-almost every curve,
namely, $u\circ \gz$ is
absolutely continuous on $[0,\,\ell(\gz)]$ for all arc-length
parameterized paths $\gz\in\Gamma_{\rm rect}\setminus \Gamma$,
where $\Gamma$ has $p$-modulus zero.
For every $\gz\in\Gamma_{\rm rect}\setminus \Gamma$,
we will show that \eqref{e4.1} holds.
To see this, by the absolute continuity of $u$ on $\gz$,
it suffices to show that for $j$ large enough,
$$2^j\lf|\int_0^{2^{-j}}u\circ \gz(t)\,dt-\int_{\ell(\gz)-2^{-j}}^{\ell(\gz)}u\circ \gz(t)\,dt\r|\ls \int_0^{\ell(\gz)}\rho_k\circ\gz(t)\,dz.$$
But,  borrowing some ideas from \cite{bp03}, for $j$ large enough, we have
\begin{eqnarray*}
&&2^j\lf|\int_0^{2^{-j}}u\circ \gz(t)\,dt-\int_{\ell(\gz)-2^{-j}}^{\ell(\gz)}u\circ \gz(t)\,dt\r|\\
&&\quad=2^j\lf|\int_0^{\ell(\gz)-2^{-j}}[u\circ \gz(t+2^{-j}) - u\circ \gz(t)]\,dt\r|\\
&&\quad\le  2^j\int_0^{\ell(\gz)-2^{-j}}\lf|u\circ \gz(t+2^{-j}) - u\circ \gz(t)\r|\,dt\\
&&\quad\ls\int_0^{\ell(\gz)-2^{-j}}\lf[\rho_k\circ \gz(t+2^{-j}) +\rho_k\circ \gz(t)\r]\,dt\\
&&\quad\ls\int_0^{\ell(\gz) } \rho_k\circ \gz(t) \,dt.
\end{eqnarray*}
This means that $\rho_k$ is a $p$-weak upper gradient of $\wz u$.

To prove the triviality of $\dot N^{1}_{p,\,q}(\cx)$ with $q\in(p,\,\fz)$,
for $u\in\dot N^{1}_{p,\,q}(\cx)$, applying Theorem \ref{t2.1}, we have 
$$\|\{I_{2^{-k}}^{1,\,1}(u)\}_{k\in\zz}\|_{\ell^q(L^p(\cx))}\sim\|u\|_{\dot N^1_{p,\,q}(\cx)}<\fz,$$ 
which 
implies that $\| I_{2^{-k}}^{1,\,1}(u) \|_{ L^p(\cx)}\to 0$ as $k\to\fz$. 
For every $k\in\zz$, let $\{x_{k,\,i}\}_{i} $ be  a maximal set of $\cx$
with $d(x_{k,\,i},\,x_{k,\,j})\ge 2^{-k-2}$ for all $i\ne j$. Then $\cb_k=\{B(x_{k,\,i},\,2^{-k})\}_i$
is a covering of $\cx$ with bounded overlap.
 Let $\{\vz_{k,\,i}\}_i$ be a partition of unity
with respect to $\cb_k$ as in \cite[Lemma 5.2]{hkt07}.
We define a discrete convolution approximation to $u$ by  
$u_{\cb_k}=\sum_i u_{B(x_{k,\,i},\,2^{-k})}\vz_{k,\,i}.$
By an argument similar to that of \cite[Lemma 5.3]{hkt07}, we have that 
$u_{\cb_k}\to u$ in $L^p_\loc(\cx)$ and hence in $L^1_\loc(\cx)$  as $k\to\fz$, 
and that $\lip\, u_{\cb_k}(x)\le C I_{2^{-k+N}}^{1,\,1}(u)(x)$ for all $x\in\cx$,
where $C\ge1$ and $N\in\nn$ are constants independent of $k$,  $x$ and $u$. 
Now $CI_{2^{-k+N}}^{1,\,1}(u)$ is an upper gradient of 
$  u_{\cb_k}$. So, for every ball $B=B(x_B,\,r_B)$, by the weak $(1,\,p)$-Poincar\'e inequality, we have 
\begin{eqnarray*}
\bint_B|u(z)-u_B|\,d\mu(z)&&=\lim_{k\to\fz}\bint_B|u_{\cb_k}(z)-(u_{\cb_k})_B|\,d\mu(z)\\
&&
\ls \liminf_{k\to\fz}r_B\lf\{\bint_{B(x_B,\,\lz r_B)} [I_{2^{-k+N}}^{1,\,1}(u)(z)]^p\,d\mu(z)\r\}^{1/p}=0, 
\end{eqnarray*}
which implies that $u$ is a constant on $B$ and hence is a constant function on $\cx$. 
This finishes the proof of Theorem \ref{t4.1}.
\end{proof}

\begin{rem}\label{r4.1}\rm
(i) Under the assumptions of Theorem \ref{t4.1},
$\dot M^{1}_{p,\,\fz}(\cx)$ is not necessarily trivial;
$\dot M^{1}_{p,\,\fz}(\rn)=\dot W^{1,\,p}(\rn)$.
Also $\dot N^{1}_{p,\,\fz}(\cx)$ is not necessarily trivial;
$\dot N^{1}_{p,\,\fz}(\rn)$ contains smooth functions with compact supports.
The argument at the end of the proof of Theorem \ref{t4.1} is due to Eero Saksman and Tom\'as Soto. 

(ii)  Theorem \ref{t4.1} when  $p=q=n$ was established in \cite{bp03}.
\end{rem}

\begin{thm}\label{t4.2}
Suppose that $\cx$   supports a
weak $(1,\,p)$-Poincar\'e inequality with $p\in(1,\,\fz)$.
Let $s\in(1,\,\fz)$.
Then for $q\in(0,\,\fz]$,
 $\dot M^s_{np/(n+ps-p),\,q}(\cx)$ is trivial, and for $q\in(0,\,np/(n+ps-p)]$,
 $\dot N^s_{np/(n+ps-p),\,q}(\cx)$ is trivial.
Moreover, if either $\cx$ is complete or $\cx$
supports a weak $(1,\,p-\ez)$-Poincar\'e inequality for some $\ez\in(0,\,p-1)$,
then for $q\in(np/(n+ps-p),\,\fz]$,
 $\dot N^s_{np/(n+ps-p),\,q}(\cx)$ is trivial.
\end{thm}

\begin{proof}
We first prove the triviality of $\dot M^s_{np/(n+ps-p),\,\fz}(\cx)=\dot M^ {s,\,np/(n+ps-p)}(\cx)$
by considering the following three cases:
Case $\mu(\cx)<\fz$, Case $\mu(\cx)=\fz$ and $\cx$ is Ahlfors $n$-regular,
 and Case $\mu(\cx)=\fz$ but $\cx$ is not
Ahlfors $n$-regular.

{\it Case  $\mu(\cx)<\fz$}. Notice that by \eqref{e2.11},    $2^{-k_0-1}\le \diam\cx< 2^{-k_0}$ for some $k_0\in\zz$.
In this case, it suffices to prove that
$\dot M^{s,\,np/(n+ps-p)}(\cx)\subset \dot M^1_{p,\,\sz}(\cx)$ for some $\sz\in(0,\,p)$; then
the triviality of $\dot M^{s,\,np/(n+ps-p)}(\cx)$ follows from that of $\dot M^1_{p,\,\sz}(\cx)$ as proved by Theorem \ref{t4.1}.
To this end, let $u\in \dot M^{s,\,np/(n+ps-p)}(\cx)$ and let $g\in\cd^s(u)$ with
$\|g\|_{L^{np/(n+ps-p)}(\cx)}\le 2\|u\|_{\dot M^{s,\,np/(n+ps-p)}(\cx)}$.
We claim that there exists $\sz\in(0,\,p)$ such that 
\begin{equation}\label{e4.2}
\|\{I^{1,\,\sz}_{2^{-k}}(u)\}_{k\ge k_0-2}\|_{L^p(\cx,\,\ell^\sz)}\ls \|g\|_{L^{np/(n+ps-p)}(\cx)}. 
\end{equation}
Assume that this claim holds for a moment. By Theorem \ref{t3.1}(ii)
and a variant of Lemma \ref{l2.5}, we have   $u\in\dot M^1_{p,\,\sz}(\cx)$
and $\|u\|_{\dot M^1_{p,\,\sz}(\cx)}\ls \|u\|_{\dot M^{s,\,np/(n+ps-p)}(\cx)}$.

To prove \eqref{e4.2}, by Lemma \ref{l2.2},
\begin{eqnarray*}
\|\{I^{1,\,\sz}_{2^{-k}}(u)(x)\}_{k\ge k_0-2}\|^{\sz}_{\ell^{\sz}}&&=
\sum_{k\ge k_0-2} 2^{k \sz}\inf_{c\in\rr}  \bint_{B(x,\,2^{-k})}|u(z)-c|^{\sz}\,d\mu(z) \\
&&\ls \sum_{k\ge k_0-2} 2^{-k(s-1)\sz}   \bint_{B(x,\,2^{-k})}[g(z)]^\sz\,d\mu(z) \\
&&\ls  \sum_{k\ge k_0-2}\frac{ 2^{-k(s-1)\sz}}{\mu(B(x,\,2^{-k}))}  \sum_{j\ge k}
 \int_{B(x,\,2^{-j})\setminus B(x,\,2^{-j-1})}[g(z)]^\sz\,d\mu(z).
\end{eqnarray*}
Notice that there exists $0<\kz\le n$ such that for $j\ge k$,
$$\mu(B(x,\,2^{-k}))\gs \mu(B(x,\,2^{-j})) 2^{-(k-j)\kz};$$
see \cite{yz09}. Choosing $\sz\in(0,\,p)$ such that $\kz-(s-1)\sz>0$, we have
\begin{eqnarray}\label{e4.3}
&&\|\{I^{1,\,\sz}_{2^{-k}}(u)(x)\}_{k\ge k_0-2}\|^\sz_{\ell^\sz}\\
&&\quad\ls  \sum_{j\ge k_0-2}\frac{1}{\mu(B(x,\,2^{-j}))}
\sum_{ k=k_0-2}^{j}  2^{-k(s-1)\sz}2^{(k-j)\kz}
 \int_{B(x,\,2^{-j})\setminus B(x,\,2^{-j-1})}[g(z)]^\sz\,d\mu(z)\nonumber \\
&&\quad\ls  \sum_{j\ge k_0-2}\frac{2^{-j(s-1)\sz}}{\mu(B(x,\,2^{-j}))}
 \int_{B(x,\,2^{-j})\setminus B(x,\,2^{-j-1})}[g(z)]^\sz\,d\mu(z) \nonumber\\
&&\quad\ls {\mathcal I}_{(s-1)\sz}(g^\sz)(x),\nonumber
\end{eqnarray}
where for $\az\in(0,\,n)$, ${\mathcal I}_\az$ denotes the {\it fractional integral} defined by
$${\mathcal I}_\az(u)(x)\equiv\int_\cx \frac{[d(x,\,y)]^\az}{\mu(B(x,\,d(x,\,y)))}u(y)\,d\mu(y).$$
Therefore,
$$\|\{I^{1,\,\sz}_{2^{-k}}(u)\}_{k\ge k_0-2}\|_{L^p(\cx,\,\ell^\sz)}
\ls \|[{\mathcal I}_{(s-1)\sz}(g^\sz)]^{1/\sz}\|_{L^p(\cx)}\sim \|{\mathcal I}_{(s-1)\sz}(g^\sz)\|^{1/\sz}_{L^{p/\sz}(\cx)}.$$
Notice that for all $x\in\cx$ and $r\le \diam\cx$,
$$\mu(B(x,\,r))\ge C\mu(\cx)\frac{r^n}{(\diam\cx)^n}\gs  r^n.$$
Recall that ${\mathcal I}_\az$ is bounded from $L^p(\cx)$ to $L^{p_\ast(\az)}(\cx)$
for all $p\in(1,\,n/\az)$; see, for example, \cite[Theorem 3.22]{hei}.
We have
$$\|\{I^{1,\,\sz}_{2^{-k}}(u)\}_{k\ge k_0-2}\|_{L^p(\cx,\,\ell^\sz)}
\ls   \| g^\sz\|^{1/\sz}_{L^{np/(n+ps-p)\sz}(\cx)}\sim
\|g\|_{L^{np/(n+ps-p)}(\cx) },$$
which  gives \eqref{e4.2}.

{\it Case $\mu(\cx)=\fz$ and $\cx$ is Ahlfors $n$-regular}. Recall that  $\cx$ is {\it Ahlfors $n$-regular}
if for all $x\in\cx$ and $r>0$,
$$\mu(B(x,\,r))\sim r^n.$$
Observe that  the fractional integral ${\mathcal I}_\az$ is still bounded from
$L^p(\cx)$ to $L^{p_\ast(\az)}(\cx)$, and hence by an argument similar to  above, we have
$\dot M^{s,\,np/(n+ps-p)}(\cx)\subset \dot M^1_{p,\,\sz}(\cx)$ for some $\sz\in(0,\,p)$,
which implies the triviality of $\dot M^{s,\,np/(n+ps-p)}(\cx)$.

{\it Case $\mu(\cx)=\fz$ but $\cx$ is not Ahlfors $n$-regular}.
Notice  that, in this case, we do not have the boundedness from
$L^p(\cx)$ to $L^{p_\ast(\az)}(\cx)$ of the fractional integral ${\mathcal I}_\az$
and hence we cannot prove
$\dot M^{s,\,np/(n+ps-p)}(\cx)\subset \dot M^1_{p,\,\sz}(\cx)$ for some $\sz\in(0,\,p)$ as above.
But the ideas of an imbedding as above and the proof of Theorem 4.1 still work here for a localized version.
Indeed, we will show that any function $u\in\dot M^{s,\,np/(n+ps-p)}(\cx)$ is constant on every ball of $\cx$,
which implies that $u$ is a constant function on whole $\cx$.

To this end, let $x_0\in\cx$, $k_0 $ be a negative integer and
let $\eta$ be a cutoff functions such that $\eta(x)=1-\dist(x,\, B(x_0,\,2^{-k_0}))$ on $B(x_0,\,2^{-k_0+1})$
and  $\eta(x)=0$ on $\cx\setminus B(x_0,\,2^{-k_0+1})$. Observe that $\eta(x)=1$ on $B(x_0,\,2^{-k_0})$.

For every $u\in\dot M^{s,\,np/(n+ps-p)}(\cx)$ with $g\in\cd^s(u)\cap L^{np/(n+ps-p)}(\cx)$,
we first claim that $u\eta\in \dot M^{1,\,p}(\cx)$.
Indeed,
%$$\chi\sum_{j\ge k_0}2^{(k_0-j)}I^{1,\,\sz}_{2^{-j}}(u)+u\chi_{B(x_0,\,2^{-k_0-2})\setminus B(x_0,\,2^{-k_0-3})}\in\cd^1(u\chi).$$
 %In fact,
if $x,\,y\in B(x_0,\,2^{-k_0+2})$,
\begin{eqnarray*}
 |u(x)\eta(x)-u(y)\eta(y)|&&=|u(x)-u(y)|\eta(x)+|u(x)||\eta(x)-\eta(y)|\\
&&\le |u(x)-u(y)|+d(x,\,y)[|u(x)|+|u(y)|]\\
&&\le d(x,\,y)[|u(x)|+|u(y)|+h(x)+
h(y)]
\end{eqnarray*}
with $h=\chi_{B(x_0,\,2^{-k_0+2})}\sum_{j\ge k_0-4}2^{(k_0-j)}I^{1,\,\sz}_{2^{-j}}(u)$;
if $x,\,y\in \cx\setminus B(x_0,\,2^{-k_0+1})$,
$u(x)\chi(x)=0=u(y)\chi(y)$;
if $x \in B(x_0,\,2^{-k_0+1})$ and
 $y\in \cx\setminus B(x_0,\,2^{-k_0+2})$ or
$y \in B(x_0,\,2^{-k_0+1})$ and
 $x\in \cx\setminus B(x_0,\,2^{-k_0+2})$,  then
by $d(x,\,y)\ge2^{-k_0}$, we have
$$|u(x)\chi(x)-u(y)\chi(y)|=|u(x)|+|u(y)|\le 2^{k_0}d(x,\,y)[|u(x)|+|u(y)|].$$
This means that $(u\chi_{B(x_0,\,2^{-k_0+2})}+h)\in \cd^1(u)$ modulo a constant depending on $k_0$.
Notice that, by Lemma \ref{l2.2},
$u\in L^p_\loc(\cx)$. So to obtain  $(u\chi_{B(x_0,\,2^{-k_0+2})}+h)\in L^p(\cx)$, it suffices to prove that
$h\in  L^p(\cx)$. For $\az\in(0,\,n)$, define the {\it local fractional integral} by
$${\mathcal J}_\az(g)(x)=\int_{d(x,\,y)\le 2^{-k_0+4}} \frac{[d(x,\,y)]^\az}{\mu(B(x,\,d(x,\,y)))}g(y)\,d\mu(y).$$
By an argument similar to that of \eqref{e4.3}, for $x\in B(x_0,\,2^{-k_0+2})$, we still have
\begin{eqnarray*}
h(x)\le \|\{I^{1,\,\sz}_{2^{-k}}(u)(x)\}_{k\ge k_0-4}\|_{\ell^\sz}\ls
\lf[{\mathcal J}_{(s-1)\sz}(g^\sz\chi_{B(x_0,\,2^{-k_0+4})})(x)\r]^{1/\sz}.
\end{eqnarray*}
Obviously,  ${\mathcal J}_{(s-1)\sz}(g^\sz\chi_{B(x_0,\,2^{-k_0+4})})$ is supported in $B(x_0,\,2^{-k_0+8})$.
Moreover,  by an argument similar to that of \cite[Theorem 3.22]{hei}, for $\az\in(0,\,n)$
one can prove that  ${\mathcal J}_\az$ is bounded from
$L^p(B(x_0,\,2^{-k_0+4}))$ to $L^{p_\ast(\az)}(\cx)$ with its operator norm depending
on $k_0$, $x_0$, $\az$ and $\cx$.
This together with an argument similar to that for the case $\mu(\cx)<\fz$  implies that $h\in L^p(\cx)$ and hence the claim that $u\eta\in \dot M^{1,\,p}(\cx)$.

For $k\in\nn$, set
$$\wz \rho_k(x)\equiv   \sup_{j\ge k} 2^{j }\lf(\inf_{c\in\rr}\bint_{B(x,\,2^{-j})}|(u\eta)(z)-c|^\sz \,d\mu(z)\r)^{1/\sz}.
$$
Since $u\eta\in \dot M^{1,\,p}(\cx)$, as what we did in the proof of Theorem \ref{t4.1},
we can show that $\wz\rho_k$ is a $p$-weak upper gradient
of $\wz {u\eta}$. Notice that $\wz {u\eta}(x)= u(x)\eta(x)=u(x)$ for almost all $x\in B(x_0,\,2^{-k_0})$,
and that for all $x\in B(x_0,\,2^{-k_0-1})$ and $j\ge k\ge k_0$,
 $(u\eta)_{B(x,\,2^{-j})}=u_{B(x,\,2^{-j})}$, and hence
$\wz \rho_k(x)=  \sup_{j\ge k}  I^{1,\,\sz}_k(u)(x).$ 
Moreover, \begin{eqnarray*}
\lf\|\lf\{I^{1,\,\sz}_{2^{-k}}(u)\r\}_{k\ge k_0}\r\|_{L^p(B(x_0,\,2^{-k_0-1}),\,\ell^\sz)}
&& \ls \lf\|\lf[{\mathcal J}_{(s-1)\sz}(g^\sz\chi_{B(x_0,\,2^{-k_0+4})})\r]^{1/\sz}\r\|_{L^p(\cx)}\\
&& \ls\|g\|_{L^{np/(n+ps-p)}(B(x_0,\, 2^{-k_0+4}))}<\fz,
\end{eqnarray*}
 which implies that $\|\{I^{1,\,\sz}_{2^{-k}}(u)(x)\}_{k\ge k_0}\|_{\ell^\sz}<\fz$  
and hence  $\wz \rho_k(x)\le \|\{I^{1,\,\sz}_{2^{-j}}(u)(x)\}_{j\ge k}\|_{\ell^\sz}\to0$ as $k\to\fz$
for almost all $x\in B(x_0,\,2^{-k_0-1})$. 
Then by the Lebesgue dominated convergence theorem, we have 
$\|\wz \rho_k\|_{L^p(B(x_0,\,2^{-k_0-1}))}\to0$ as $k\to\fz$. 
Applying the Poincar\'e inequality, we obtain
\begin{eqnarray*}
&&\inf_{c\in\rr}\bint_{B(x_0,\,2^{-k_0-1}/\lz)}|u(z)-u_{ B(x,\,2^{-k_0-1}/\lz)}| \,d\mu(z)\\
&&\quad=\inf_{c\in\rr}\bint_{B(x_0,\,2^{-k_0-1}/\lz)}|(u\eta)(z)-(u\eta)_{B(x,\,2^{-k_0-1}/\lz)}| \,d\mu(z)\\
&&\quad\ls\lf( \bint_{B(x_0,\,2^{-k_0-1})}\wz \rho_k^p(z) \,d\mu(z)\r)^{1/p}\to 0.
\end{eqnarray*}
This means that
$u$ is a constant on $B(x,\,2^{-k_0-1}/\lz)$.
Since $k_0$ is arbitrary, we conclude that $u$ is a constant function.

Moreover, for   $q\in(0,\,\fz]$,  the triviality of  $\dot M^s_{np/(n+ps-p),\,q}(\cx)$ follows from
$$\dot M^s_{np/(n+ps-p),\,q}(\cx)\subset \dot M^{s,\,np/(n+ps-p)}(\cx).$$
Meanwhile, for   $q\in(0,\,np/(n+ps-p)]$,
the triviality of $\dot N^s_{np/(n+ps-p),\,q}(\cx)$ follows from
$$\dot N^s_{np/(n+ps-p),\,q}(\cx)\subset \dot M^s_{np/(n+ps-p),\,np/(n+ps-p)}(\cx)\subset
\dot M^{s,\,np/(n+ps-p)}(\cx).$$

Finally, we prove  the triviality of $\dot N^s_{np/(n+ps-p),\,q}(\cx)$ for $q\in(np/(n+ps-p),\,\fz]$.
In fact, it follows from the triviality of $\dot N^s_{np/(n+ps-p),\,\fz}(\cx)$ since
$\dot N^s_{np/(n+ps-p),\,q}(\cx)\subset  \dot N^s_{np/(n+ps-p),\,\fz}(\cx)$.
To see the triviality of $\dot N^s_{np/(n+ps-p),\,\fz}(\cx)$,
we need the additional  condition that $\cx$ supports a weak $(1,\,p-\ez)$-Poincar\'e inequality
for some $\ez\in(0,\,p-1)$. Recall  from \cite{kz08} that
if $\cx$ is complete and supports the weak $(1,\,p)$-Poincar\'e inequality,
then $\cx$ supports a weak $(1,\,p-\ez)$-Poincar\'e inequality
for some $\ez\in(0,\,p-1)$.  Without loss of generality, we can ask $\ez$ close to $0$
such that $$t\equiv s+\frac np- \frac n{p-\ez}>1.$$
Observe that
\begin{equation} \label{e4.4}
 \frac{np}{ n+p(s-1)}=\frac{n(p-\ez)}{n+(p-\ez)(t-1)}.
\end{equation}
Now we will consider the following two cases:
$\mu(\cx)<\fz$ and $\mu(\cx)=\fz$.

{\it Case $\mu(\cx)<\fz$.}  Assume that $2^{-k_0-1}\le\diam\cx< 2^{-k_0}$ for some $k_0\in\zz$.
We claim that $\dot N^s_{np/(n+ps-p),\,\fz}(\cx)\subset \dot M^{t}_{np/(n+ps-p),\,\fz}(\cx)$
for any $t\in(1,\,s)$.
Indeed, for every $u\in \dot N^s_{np/(n+ps-p),\,\fz}(\cx)$,
\begin{equation} \label{e4.5}
 \lf\|\sup_{k\ge k_0-2} I^{t,\,\sz}_{2^{-k}}(u)\r\|_{L^p(\cx)}
 \sim\lf\|\sup_{k\ge k_0-2} 2^{-k(s-t)}I^{s,\,\sz}_{2^{-k}}(u)\r\|_{L^p(\cx)}
\ls \lf\| S^{s,\,t,\,\sz}_{2^{-k_0+2}}(u)\r\|_{L^p(\cx)}.
\end{equation}
Since
$$\|u\|_{\dot N^s_{np/(n+ps-p),\,\fz}(\cx)}\sim\sup_{k\ge k_0-2}\lf\| S^{s,\,t,\,\sz}_{2^{-k}}(u)\r\|_{L^p(\cx)}$$
and
$$\|u\|_{\dot M^t_{np/(n+ps-p),\,\fz}(\cx)}\sim
 \lf\|\sup_{k\ge k_0-2} I^{t,\,\sz}_{2^{-k}}(u)\r\|_{L^p(\cx)},$$
we conclude that
$\|u\|_{\dot M^t_{np/(n+ps-p),\,\fz}(\cx)}\ls  \|u\|_{\dot N^s_{np/(n+ps-p),\,\fz}(\cx)} $
and hence our claim.
Then the triviality of $\dot N^s_{np/(n+ps-p),\,\fz}(\cx)$ follows from that of
$\dot M^{t}_{n(p-\ez)/[n+(p-\ez)(t-1)],\,\fz}(\cx)$ and \eqref{e4.4}.

{\it Case $\mu(\cx)=\fz$.}
Since the constant in \eqref{e4.5} depends on $k_0$ and hence the diameter of $\cx$,
 we can not get the imbedding
 $\dot N^s_{np/(n+ps-p),\,\fz}(\cx)\subset \dot M^{t}_{np/(n+ps-p),\,\fz}(\cx)$
for  $t\in(1,\,s)$.
But for any fixed $x_0\in\cx$ and $k_0\in\zz$,
we still have
\begin{eqnarray*}
 \lf\|\sup_{k\ge k_0-16} I^{t,\,\sz}_{2^{-k}}(u)\r\|_{L^p(B(x_0,\,2^{-k_0+8}))}
&&\sim\lf\|\sup_{k\ge k_0-16} 2^{-k(s-t)}I^{s,\,\sz}_{2^{-k}}(u)\r\|_{L^p(B(x_0,\,2^{-k_0+8}))}\\
&&
\ls \lf\| S^{s,\,t,\,\sz}_{2^{-k_0+16}}(u)\r\|_{L^p(B(x_0,\,2^{-k_0+8}))}<\fz,\nonumber
\end{eqnarray*}
which further means that $u\in M^{t,\, n(p-\ez)/[n+(p-\ez)(t-1)]}(B(x_0,\,2^{-k_0+8}))$.
With the weak $(1,\,p-\ez)$-Poincar\'e inequality in hand,
by adapting the arguments in {\it Case $\mu(\cx)$ but $ \cx$ is not Ahlfor $n$-regular} as above,
we still can prove that $u$ is constant on ball $B(x_0,\,2^{k_0-1}/\lz)$.
Hence $u$ is a constant function on whole $\cx$. We omit the details.
This finishes the proof of Theorem \ref{t4.2}.
\end{proof}

Finally, we give an example to show  the ``necessity'' of the weak $(1,\,n)$-Poincar\'e inequality
to ensure the triviality of $\dot B^s_{n/s,\,n/s}(\cx)$ for $s\in[1,\,\fz)$.
\begin{thm}\label{t4.3}
For each $p\in(2,\,\fz)$,
there exists an Ahlfors $2$-regular space $\cx$ such that $\cx$ supports a weak $(1,\,p)$-Poincar\'e inequality
but for every $s\in(0,\,\fz)$, $\dot B^s_{2/s,\,2/s}(\cx)$ is not trivial.
\end{thm}

\begin{proof}
Let $\az\in(0,\,1)$ and
$E_\az$ be the cantor set in $[0,\,1]$ obtained by first removing an interval of length $1-\az$
and leaving two intervals of length $\az/2$ and then continuing inductively.
The Hausdorff dimension $d_\az$ of $E_\az$ is $\log 2/\log(2/\az)$.
The space $\cx_\az$ is obtained by replacing each of the complementary intervals of $E_\az$
by a closed square having that interval as one of its diagonals. Then
$\cx_\az$ is Ahlfors 2-regular with respect to Euclidean distance and by \cite[Theorem 3.1]{km98},
 for any $$p>\frac{2-d_\az}{1-d_\az}=2+\frac{\log2}{-\log\az},$$
$\cx_\az$ supports the $(1,\,p)$-Poincar\'e inequality.

So for any $p>2$, choosing $\az\in(0,\,2^{-1/(p-2)})$, we know that $\cx_\az$ supports the weak $(1,\,p)$-Poincar\'e inequality.
Moreover, for any $x=(x_1,\,x_2)\in\cx_\az$, define the Cantor function by
$u(x)=\ch^{d_\az}([0,\,x_1]\cap E_\az)$. Then $u$ is constant on each square generating $\cx_\az$ and moreover,
$|u(x)-u(y)|\le |x_1-y_1|^{d_\az}\ls[d(x,\,y)]^{d_\az}$ for all $x,\,y\in\cx_\az$ (see \cite{hkt07}).
For $s>d_\az$, taking $g(x)=2[d(x,\,E_\az)]^{d_\az-s}$ for all $x\in\cx_\az$, we have $g\in \cd^s(u)$.
We claim that $g\in L^{q}(\cx_\az)$ if  $$0<q<\frac{2-d_\az}{s-d_\az}=\frac{\log 2-2\log \az}{(s-1)\log 2-s\log\az}.$$
Indeed, on each square $Q\subset\cx_\az$ with diagonal length $2^{-j}\az^j(1-\az)$, we have
$$\int_Q [g(x)]^q\,dx\ls  [2^{-j}\az^{j}]^{(d_\az-s)q+2}$$
since $(d_\az-s)q+1>-1$, namely, $q<2/(s-d_\az)$ which is given by $q<(2-d_\az)/(s-d_\az).$
Observing that there are $2^j$   such squares, we   have
 $$\int_{\cx_\az} [g(x)]^q\,dx\ls  \sum_{j\ge 1}2^j[2^{-j}\az^{j}]^{(d_\az-s)q+2}\ls
 \sum_{j\ge 1} 2^{j-j[(d_\az-s)q+2](1-\log\az/\log2)}<\fz,$$
where in the last inequality we use
$$1-[(d_\az-s)q+2](1-\log\az/\log2)= 1-[(d_\az-s)q+2]/d_\az<0,$$
which is equivalent to
$q<(2-d_\az)/(s-d_\az).$
Thus $u\in \dot M^{s,\,q}(\cx_\az)$.
In particular,  taking $q=2/s$ for each $s\in(d_\az,\,\fz)$, we know that $\dot M^{s,\,2/s}(\cx_\az)$
are nontrivial. Notice that $\dot M^{s,\,2/s}(\cx_\az)\subset  \dot B^{1}_{2 ,\,2}(\cx_\az)$ when $s>1$.
Similarly, when $0<s<1$, $\dot M^{1,\,2}(\cx_\az)\subset  \dot B^{s}_{2/s ,\,2/s}(\cx_\az)$, and moreover,
$\dot B^{s}_{2/s ,\,2/s}(\cx_\az)$ contains the restriction of any function in $\dot B^{s}_{2/s ,\,2/s}(\rn)$ to
$\cx_\az$. Then
$\dot B^{s}_{2/s,\,2/s}(\cx_\az)$ for all $s\in(0,\,\fz)$ are nontrivial. This finishes the proof of Theorem \ref{t4.3}.
\end{proof}

\begin{rem}\label{r4.2}\rm
In the proof of Theorem \ref{t4.3},  we actually proved that $\dot M^{s,\,q}(\cx_\az)$ are nontrivial for $s\in(1,\,\fz)$ and
$$0<q<\frac{2-d_\az}{s-d_\az}.$$ But $(2-d_\az)/(s-d_\az)$ is not critical by Theorem \ref{t4.2}.
So it would be interesting to know whether $ \dot M^{s,\,q}(\cx_\az)$ is trivial or not for $s\in(1,\,\fz)$ and
$$ \frac{2-d_\az}{s-d_\az}\le q\le \frac{2(2-d_\az)}{2s-(s+1)d_\az}\equiv
\frac{2\frac{2-d_\az}{1-d_\az}}{2+(s-1)\frac{2-d_\az}{1-d_\az}},$$
where the last index is  critical  by Theorem \ref{t4.2}.
\end{rem}

\medskip

\noindent{\rm\bf Acknowledgement.}
We would like to thank Eero Saksman and Tom\'as Soto for an argument that allowed us to handle all 
values of $q$ for $\dot N^{1}_{p,\,q}(\cx)$ in Theorem \ref{t4.1}.

\noindent Amiran Gogatishvili

\medskip

\noindent Institute of Mathematics, Academy of Sciences of the Czech Republic,
\v Zitn\'a 25, 115 67 Prague 1, Czech Republic

\smallskip

\noindent{\it E-mail address}:   \texttt{gogatish@math.cas.cz}

\bigskip

\noindent Pekka Koskela

\medskip

\noindent Department of Mathematics and Statistics,
P. O. Box 35 (MaD),
FI-40014, University of Jyv\"askyl\"a,
Finland
\smallskip

\noindent{\it E-mail address}:   \texttt{pkoskela@maths.jyu.fi}

\bigskip

\noindent Yuan Zhou

\medskip

\noindent Department of Mathematics, Beijing University of Aeronautics and Astronautics, Beijing 100083, P. R. China

and 

\noindent Department of Mathematics and Statistics,
P. O. Box 35 (MaD),
FI-40014, University of Jyv\"askyl\"a,
Finland

\smallskip

\noindent{\it E-mail address}:  \texttt{yuan.y.zhou@jyu.fi}

\end{document}